\documentclass[a4paper,10pt]{amsart}
\usepackage{amsmath,amsthm,amssymb,latexsym,enumerate,xcolor,hyperref}
\usepackage{graphicx}
\usepackage[margin=2.6cm,top=4.5cm, bottom=4.5cm]{geometry}

\numberwithin{equation}{section}

\begin{document}

	\newtheorem{thm}{Theorem}[section]
	\newtheorem{prop}[thm]{Proposition}
	\newtheorem{lem}[thm]{Lemma}
	\newtheorem{cor}[thm]{Corollary}
	\newtheorem{rem}[thm]{Remark}
	\newtheorem*{defn}{Definition}

	\newtheorem{definit}[thm]{Definition}
	\newtheorem{setting}{Setting}
	\renewcommand{\thesetting}{\Alph{setting}}
	
	\newcommand{\DD}{\mathbb{D}}
	\newcommand{\NN}{\mathbb{N}}
	\newcommand{\ZZ}{\mathbb{Z}}
	\newcommand{\QQ}{\mathbb{Q}}
	\newcommand{\RR}{\mathbb{R}}
	\newcommand{\CC}{\mathbb{C}}
	\renewcommand{\SS}{\mathbb{S}}

	\renewcommand{\theequation}{\arabic{section}.\arabic{equation}}

	\newcommand{\supp}{\mathop{\mathrm{supp}}}    
	
	\newcommand{\re}{\mathop{\mathrm{Re}}}   
	\newcommand{\im}{\mathop{\mathrm{Im}}}   
	\newcommand{\dist}{\mathop{\mathrm{dist}}}  
	\newcommand{\link}{\mathop{\circ\kern-.35em -}}
	\newcommand{\spn}{\mathop{\mathrm{span}}}   
	\newcommand{\ind}{\mathop{\mathrm{ind}}}   
	\newcommand{\rank}{\mathop{\mathrm{rank}}}   
	\newcommand{\Fix}{\mathop{\mathrm{Fix}}}   
	\newcommand{\codim}{\mathop{\mathrm{codim}}}   
	\newcommand{\conv}{\mathop{\mathrm{conv}}}   
	\newcommand{\epsi}{\mbox{$\varepsilon$}}
	\newcommand{\eps}{\mathchoice{\epsi}{\epsi}
		{\mbox{\scriptsize\epsi}}{\mbox{\tiny\epsi}}}
	\newcommand{\cl}{\overline}
	\newcommand{\pa}{\partial}
	\newcommand{\ve}{\varepsilon}
	\newcommand{\zi}{\zeta}
	\newcommand{\Si}{\Sigma}
	\newcommand{\cA}{{\mathcal A}}
	\newcommand{\cG}{{\mathcal G}}
	\newcommand{\cH}{{\mathcal H}}
	\newcommand{\cI}{{\mathcal I}}
	\newcommand{\cJ}{{\mathcal J}}
	\newcommand{\cK}{{\mathcal K}}
	\newcommand{\cL}{{\mathcal L}}
	\newcommand{\cN}{{\mathcal N}}
	\newcommand{\cR}{{\mathcal R}}
	\newcommand{\cS}{{\mathcal S}}
	\newcommand{\cT}{{\mathcal T}}
	\newcommand{\cU}{{\mathcal U}}
	\newcommand{\OM}{\Omega}
	\newcommand{\B}{\bullet}
	\newcommand{\ol}{\overline}
	\newcommand{\ul}{\underline}
	\newcommand{\vp}{\varphi}
	\newcommand{\AC}{\mathop{\mathrm{AC}}}   
	\newcommand{\Lip}{\mathop{\mathrm{Lip}}}   
	\newcommand{\es}{\mathop{\mathrm{esssup}}}   
	\newcommand{\les}{\mathop{\mathrm{les}}}   
	\newcommand{\nid}{\noindent}
	\newcommand{\pzr}{\phi^0_R}
	\newcommand{\pir}{\phi^\infty_R}
	\newcommand{\psr}{\phi^*_R}
	\newcommand{\pow}{\frac{N}{N-1}}
	\newcommand{\ncl}{\mathop{\mathrm{nc-lim}}}   
	\newcommand{\nvl}{\mathop{\mathrm{nv-lim}}}  
	\newcommand{\la}{\lambda}
	\newcommand{\La}{\Lambda}    
	\newcommand{\de}{\delta}    
	\newcommand{\fhi}{\varphi} 
	\newcommand{\ga}{\gamma}    
	\newcommand{\ka}{\kappa}   
	
	\newcommand{\core}{\heartsuit}
	\newcommand{\diam}{\mathrm{diam}}

	\newcommand{\lan}{\langle}
	\newcommand{\ran}{\rangle}
	\newcommand{\tr}{\mathop{\mathrm{tr}}}
	\newcommand{\diag}{\mathop{\mathrm{diag}}}
	\newcommand{\dv}{\mathop{\mathrm{div}}}
	
	\newcommand{\al}{\alpha}
	\newcommand{\be}{\beta}
	\newcommand{\Om}{\Omega}
	\newcommand{\na}{\nabla}
	
	\newcommand{\cC}{\mathcal{C}}
	\newcommand{\cM}{\mathcal{M}}
	\newcommand{\nr}{\Vert}
	\newcommand{\De}{\Delta}
	\newcommand{\cX}{\mathcal{X}}
	\newcommand{\cP}{\mathcal{P}}
	\newcommand{\om}{\omega}
	\newcommand{\si}{\sigma}
	\newcommand{\te}{\theta}
	\newcommand{\Ga}{\Gamma}
	
	\newcommand{\vV}{\mathbf{v}}
	\newcommand{\lbunu}{\ul{m}}
	\newcommand{\ca}{\ul{a}}
	\newcommand{\Vve}{\ul{\varepsilon}}
	\newcommand{\ur}{\ul{r}}
	\newcommand{\cB}{{\mathcal B}}
	
	\title[Bubbling and quantitative stability for Alexandrov's SBT with $L^1$-type deviations]{Bubbling and quantitative stability for Alexandrov's Soap Bubble Theorem with $L^1$-type deviations}

	\author{Giorgio Poggesi}
	\address{Discipline of Mathematical Sciences, The University of Adelaide, Adelaide SA 5005, Australia}
	%
	%
	%
	
	\email{giorgio.poggesi@adelaide.edu.au}

	
	\begin{abstract}
		The quantitative analysis of bubbling phenomena for almost constant mean curvature boundaries is an important question having significant applications in various fields including capillarity theory and the study of mean curvature flows. Such a quantitative analysis was initiated in [G. Ciraolo and F. Maggi, Comm. Pure Appl. Math. (2017)], where the first quantitative result of proximity to a set of disjoint balls of equal radii was obtained in terms of a uniform deviation of the mean curvature from being constant. Weakening the measure of the deviation in such a result is a delicate issue that is crucial in view of the applications for mean curvature flows. Some progress in this direction was recently made in [V. Julin and J. Niinikoski, Anal. PDE (2023)], where $L^{N-1}$-deviations are considered for domains in $\RR^N$. In the present paper we significantly weaken the measure of the deviation, obtaining a quantitative result of proximity to a set of disjoint balls of equal radii for the following deviation
		$$
		\int_{\pa\Om} \left(H_0-H\right)^+ dS_x, \quad \text{ where } 
		\begin{cases}
		H \text{ is the mean curvature of } \pa\Om ,
		\\
		H_0:=\frac{|\pa\Om|}{N |\Om|}  ,
		\\
		\left(H_0 - H\right)^+:=\max\left\lbrace H_0 - H , 0 \right\rbrace ,
	    \end{cases}
		$$
		which is clearly even weaker than $\nr H_0-H \nr_{L^1(\pa\Om)}$.
	\end{abstract}


	\maketitle

	\raggedbottom

	\section{Introduction}
	
	Alexandrov's Soap Bubble Theorem states that if a (smooth) bounded domain\footnote{In our notation, domain means connected open set.} $\Om$ in $\RR^N$ ($N\ge 2$) has boundary $\pa\Om$ with constant mean curvature, then $\Om$ is a ball (and $\pa\Om$ is a sphere).
	In other words, such a rigidity statement provides the characterization of (smooth)\footnote{See \cite{DM} for the corresponding generalization to sets of finite perimeter.}
	critical points of the isoperimetric problem.
	The quantitative stability analysis of such a rigidity result is fundamental for various applications in several research areas including capillarity theory (see, e.g., \cite{CirMag}) and the study of mean curvature flows (see, e.g., \cite{JN}). In particular, for the last type of applications it is important to consider weak-type deviations of the mean curvature from being constant.
	
	Notice that, given any $R>0$, disconnected sets made of disjoint balls of radius $R$ still satisfy $H \equiv 1/R$ (where $H$ denotes the mean curvature), but such type of sets are excluded in the rigidity statement above due to the connectedness assumption on $\Om$. Nevertheless, when studying the stability issue of such a rigidity statement, configurations of connected sets converging to disjoint balls (of equal radii) can display small deviations of their mean curvature from being constant and hence must be taken into account. These are the so-called {\it bubbling} phenomena, which in this context have been quantitatively studied in \cite{CirMag,JN} (see also the seminal papers \cite{BC,St} which deal with parametrized surfaces and \cite{DMMN} where a qualitative but not quantitative analysis is provided in a very general setting).

	Under assumptions preventing bubbling phenomena (e.g., a uniform sphere condition or isoperimetric-type bounds), several quantitative stability results of proximity to a single ball have been established: \cite{CV,KM,MP,MP2,MP3,MP6,Scheuer}. Such a case is now well understood and, in fact, optimal rates of stability have been obtained (for various measures of the deviation of the mean curvature from being constant) in \cite{CV} (for uniform deviations), \cite[Theorem 1.10]{KM} (restricting the analysis to nearly spherical sets but allowing $L^p$-deviations, with $p\ge 2$ for $N \le 4$ and $p>(N-1)/2$ for $N \ge 5$), \cite[Theorem 4.8]{MP2} and \cite[(i) of Theorem 1.3]{Pog} (for $L^2$-deviations in any dimension); see also \cite[Theorem 3.9]{MP6}, where an optimal linear estimate was obtained for the Gauss map of $\pa \Om$ (for $L^2$ deviations in any dimension). For domains satisfying a uniform interior sphere condition, sharp quantitative stability has also been obtained for the Heintze-Karcher inequality in \cite{Pog}; such an inequality \cite{HK} deals with mean convex domains and is strictly related to Alexandrov's Soap Bubble Theorem, as highlighted by the proofs of Alexandrov's Soap Bubble Theorem discovered by Ros \cite{Ros} and Montiel and Ros \cite{MR}. Other possibly non sharp quantitative results of closeness to a single ball with $L^1$-deviations were obtained in \cite{MP,MP3,Scheuer,ScheuerXia}.
	
	Despite the huge literature on the quantitative stability for the Soap Bubble Theorem providing results of proximity to a single ball, to the author knowledge only two papers deal with general quantitative stability results (allowing bubbling phenomena): the first is \cite{CirMag}, which deals with uniform (i.e., $C^0$) deviations, and the second is \cite{JN} which deals with $L^{N-1}$-deviations.
	Weakening the measure of the deviation in such a type of results is a delicate issue
	that is particularly important for possible applications to research related to mean curvature flows (see, e.g., \cite{JN}, which was in fact motivated by such a type of applications).
	In the present paper we significantly weaken the measure of the deviation, obtaining a quantitative result of proximity to a set of disjoint balls of equal radii for $L^1$-type deviations.
	
	Our approach is completely different to those used in \cite{CirMag,JN}. In fact, \cite{CirMag} is based on Ros' proof \cite{Ros} of the Soap Bubble Theorem, \cite{JN} is based on the proof given in \cite{MR}, whereas our proof is based on the proof established by Reilly in \cite{Reilly}. As a matter of fact, some of our arguments follow the spirit of \cite{BNST}, where quantitative stability (including bubbling phenomena) was established for Serrin's overdetermined problem (\cite{Se,We}). Hence, our result answers to a question posed  in \cite{CirMag} where the authors noticed that it was not clear if one could take advantage of \cite{BNST} in achieving corresponding results for the Soap Bubble Theorem. In fact, as noticed in \cite{CirMag}, the presence in \cite{BNST} of a technical pointwise assumption on the gradient of the torsion function does not allow to leverage \cite{BNST} in the quantitative study of the Soap Bubble Theorem, but a careful new analysis is required, which we provide here.
	
	We finally mention that from a qualitative point of view, a result of convergence for uniformly bounded and uniformly mean convex sequences of domains with $L^2$-almost constant mean curvature was established in \cite[Theorem 1.1]{DMMN}. Our result provides a quantitative refinement of that result\footnote{The result in \cite{DMMN} also treats the anisotropic setting, which is not treated in the present paper. Nevertheless, the techniques used here may be extended to the anisotropic setting; such an extension will be treated in forthcoming research.}, weakens the mean-convexity assumption, and relaxes the $L^2$-deviation with an $L^1$-type deviation. Actually, as a deviation we consider
	\begin{equation}\label{eq:def delta}
	\de := \int_{\pa\Om} \left(H_0-H\right)^+ dS_x, \quad \text{ where } H_0:=\frac{|\pa\Om|}{N |\Om|} \text{ and } \left(H_0 - H\right)^+(x):=\max\left\lbrace H_0 - H(x) , 0 \right\rbrace,
    \end{equation} 
	which is clearly even weaker than the standard $L^1$-deviation
	$$
	\nr H_0-H \nr_{L^1(\pa\Om)} .
	$$
	Here and in the following, $|\Om|$ and $| \pa \Om|$ denote the $N$-dimensional Lebesgue measure of $\Om$ and the surface measure of $\pa \Om$.
	
	In what follows, we set
	\begin{equation}\label{eq:alpha}
	\al:= \frac{2}{2N+7} ,
	\end{equation}
	and denote with $|B_1|$ the volume of the unit ball in $\RR^N$.
	\begin{thm}\label{thm:stability with diameter}
	Let $\Om \subset \RR^N$, $N\ge 2$, be a bounded open set of class $C^2$ such that
	\begin{equation}\label{eq:assumption lower bound M_0-}
		\infty > \cM_0^-:= \max_{\pa\Om} H^- ,
	\end{equation}
	where $H^- (x):= \max\left\lbrace - H(x) , 0 \right\rbrace$ denotes the negative part of the mean curvature $H$ of $\pa\Om$.
	Set 
	\begin{equation}\label{eq:def H_0 and R}
	R:= \frac{N|\Om|}{|\pa\Om|} \quad \text{and} \quad H_0:=\frac{1}{R} ,
	\end{equation}
	assume that\footnote{The assumption $R \ge 1$ is always satisfied up to a dilation.} $R\ge1$, and consider the deviation $\de$ defined in \eqref{eq:def delta}.
	
	Then, there exist $m$ points $z_i \in \Om$ and $m$ disjoint balls $B_{\rho_{int}^i} (z_i) \subset \Om$ (centered at $z_i$, with radius $\rho_{int}^i$, and contained in $\Om$), $i = 1, \dots, m$, such that,
	setting 
	\begin{equation*}
		\mathcal{F} := \bigcup_{i=1}^m B_{\rho_{int}^i} (z_i) ,
	\end{equation*} 
	we have that	
	\begin{equation}\label{eq:thm_new global statement radii}
			| \rho_{int}^i - R| \le C \,  \de^{\frac{\al}{4}} \qquad \text{for any } i = 1, \dots, m,
	\end{equation}
	\begin{equation}\label{eq:thm_new global statement measure}
		\left| \Om \De \mathcal{F}  \right| \le C \,  \de^{ \frac{\al}{ \max \left\lbrace 4 , N \right\rbrace }} ,
	\end{equation}
	\begin{equation}\label{eq:thm_one-sided Hausdorff closeness}
			\max_{x\in \pa \mathcal{F} } \mathrm{dist}_{\pa\Om} (x) \le C \, \de^{\frac{\al}{2}},
	\end{equation}
	\begin{equation}\label{eq:thm_perimeter closeness}
			\left| | \pa \Om| - \left| \pa \mathcal{F} \right| \right| \le C \,  \de^{ \frac{\al}{ \max \left\lbrace 4 , N \right\rbrace }} ;
	\end{equation}
	moreover, the number $m$ is bounded from above by means of
	\begin{equation}\label{eq:bound on number m}
		m \le \frac{|\Om|}{|B_1| | R - C \, \de^{\frac{\al}{4}} |^N} .
	\end{equation}
	The constants $C$ appearing in \eqref{eq:thm_new global statement radii}, \eqref{eq:thm_new global statement measure}, \eqref{eq:thm_one-sided Hausdorff closeness}, \eqref{eq:thm_perimeter closeness}, and \eqref{eq:bound on number m} (can be explicitly computed and) only depend on the dimension $N$, and upper bounds on the diameter $d_\Om$ of $\Om$ and $\cM_0^-$.
	\end{thm}

%
%

	As already mentioned, our result strongly weakens the measure of the deviation (of the mean curvature from being constant) considered in the previous quantitative stability results (i.e., \cite{CirMag,JN}). In fact, we consider an $L^1$-type deviation, whereas \cite{CirMag} considered uniform deviations (i.e., $C^0$-deviations) and \cite{JN} considered $L^{N-1}$-deviations.
	We also relax the mean-convexity assumption that was required in \cite{CirMag,DMMN}: see (ii) of Remark \ref{rem:intro}.
	Finally, the direct nature of our proof may open the doors to improvements and refinements possibly leading to the optimal stability exponent, which, in presence of bubbling, is a widely open question.

	\begin{rem}\label{rem:intro}
{\rm	
	\begin{enumerate}[(i)]
	\item If in Theorem \ref{thm:stability with diameter} $\Om$ is assumed to be connected, recalling Topping's inequality \cite{Topping}
	\begin{equation}
	d_\Om \le C(N) \int_{\pa\Om} |H|^{N-2} \, dS_x , \text{ where } C(N) \text{ is a constant only depending on } N,
	\end{equation}
	the dependence of the constants in Theorem \ref{thm:stability with diameter} on the (upper bound on the) diameter can be replaced with that on 
	$$
	\int_{\pa\Om} |H|^{N-2} \, dS_x .
	$$
	Such a dependence can be totally removed if the deviation of the mean curvature $H$ from the reference constant $H_0$ is measured in $L^{N-2}$.
	
	\item Assumption \eqref{eq:assumption lower bound M_0-} only requires a finite upper bound on the negative part of $H$; it is clear that if one restricts the analysis to mean-convex domains (as in \cite{CirMag,DMMN}), then the dependence of the constants in Theorem \ref{thm:stability with diameter} from $\cM_0^-$ can be removed (as we can take $\cM_0^-=0$).
	We also stress that assumption \eqref{eq:assumption lower bound M_0-} is not required in \cite{JN} (where $L^{N-1}$-deviations are considered), but is necessary in our setting, as it can be directly checked by considering $\Om:= B_2 \setminus \ol{B}_\epsilon$ for $\epsilon \to 0$ (where $B_2$ and $B_\epsilon$ denote the balls with radii $2$ and $\epsilon$ centered at the origin); in fact, being as $H=1/2$ on $\pa B_2$ and $H= - 1/\epsilon$ on $\pa B_\epsilon$, for $\epsilon \to 0$ we have that \eqref{eq:assumption lower bound M_0-} is not satisfied, and $\de \to 0$ but the conclusion of Theorem \ref{thm:stability with diameter} clearly fails.
    \end{enumerate}
}
\end{rem}
The consequences of the above Remark are collected in the following Corollary.

\begin{cor}\label{cor:two corollaries}
(i) Let $\Om \subset \RR^N$, $N\ge 2$, be a bounded domain of class $C^2$ such that \eqref{eq:assumption lower bound M_0-} holds along with
\begin{equation}
\infty > \cT:=\int_{\pa\Om} H^{N-2} \, dS_x .
\end{equation}
Let $R$ and $H_0$ be those defined in \eqref{eq:def H_0 and R}, and assume that $R \ge 1$.
Then, the conclusion of Theorem~\ref{thm:stability with diameter} holds true with (explicit) constants $C$ (in \eqref{eq:thm_new global statement radii}, \eqref{eq:thm_new global statement measure}, \eqref{eq:thm_one-sided Hausdorff closeness}, \eqref{eq:thm_perimeter closeness}, and \eqref{eq:bound on number m})
only depending on $N$, and upper bounds on $\cM_0^-$ and $\cT$.
If $\Om$ is mean convex (i.e., $H \ge 0$ on $\pa\Om$), then the dependence on $\cM_0^-$ can be removed.

(ii) Let $\Om \subset \RR^N$, $N\ge 3$, be a bounded domain of class $C^2$ such that $|\pa\Om| \le C_0$ and $|\Om| \ge 1/C_0$. Let $R$ and $H_0$ be those in \eqref{eq:def H_0 and R}. If $\Om$ is mean convex and
\begin{equation}\label{eq:nuovaaggiuntasmallnessincorollary}
\nr H_0 - H \nr_{L^{N-2}(\pa\Om)} \le c ,
\end{equation}
then
there exist $m$ points $z_i \in \Om$ and $m$ disjoint balls $B_{\rho_{int}^i } (z_i) \subset \Om$ (centered at $z_i$, with radius $\rho_{int}^i$, and contained in $\Om$), $i = 1, \dots, m$, such that,
setting 
\begin{equation*}
	\mathcal{F} := \bigcup_{i=1}^m B_{\rho_{int}^i} (z_i) ,
\end{equation*} 
we have that
\begin{equation}\label{eq:stab in L^N-2 I}
	| \rho_{int}^i - R| \le C \,  \nr H_0 - H \nr_{L^{N-2}(\pa\Om)}^{\frac{\al}{4}} \qquad \text{for any } i = 1, \dots, m,
\end{equation}
\begin{equation}\label{eq:stab in L^N-2 II}
	\left| \Om \De 	\mathcal{F} \right| \le C \,  \nr H_0 - H \nr_{L^{N-2}(\pa\Om)}^{ \frac{\al}{ \max \left\lbrace 4 , N \right\rbrace }} ,
\end{equation}
\begin{equation}\label{eq:stab in L^N-2_IV_one-sided Hausdorff closeness}
	\max_{x\in \pa \mathcal{F} } \mathrm{dist}_{\pa\Om} (x) \le C \, \nr H_0 - H \nr_{L^{N-2}(\pa\Om)}^{\frac{\al}{2}},
\end{equation}
\begin{equation}\label{eq:stab in L^N-2_V_perimeter closeness}
	\left| | \pa \Om| - \left| \pa \mathcal{F} \right| \right| \le C \,  \nr H_0 - H \nr_{L^{N-2}(\pa\Om)}^{ \frac{\al}{ \max \left\lbrace 4 , N \right\rbrace }} ;
\end{equation}
moreover, the number $m$ is bounded from above by means of
\begin{equation}\label{eq:stab in L^N-2 III}
	m \le \frac{|\Om|}{|B_1| \left( R - C \, \nr H_0 - H \nr_{L^{N-2}(\pa\Om)}^{\frac{\al}{4}} \right)^N} .
\end{equation}
Here, $C$ (appearing in \eqref{eq:stab in L^N-2 I},
\eqref{eq:stab in L^N-2 II}, \eqref{eq:stab in L^N-2_IV_one-sided Hausdorff closeness}, \eqref{eq:stab in L^N-2_V_perimeter closeness}, and \eqref{eq:stab in L^N-2 III}) and $c$ (appearing in \eqref{eq:nuovaaggiuntasmallnessincorollary}) are
(explicit) constants 
only depending on the dimension $N$ and $C_0$.
%
\end{cor}

The assumption 
$$|\pa\Om| \le C_0 \text{ and } |\Om| \ge 1/C_0$$ 
(which clearly implies $R\ge N / C_0^2$)
is taken in (ii) of the above corollary following the spirit of \cite[Theorem 1.2]{JN}. It is easy to check that such an assumption may be replaced with an upper bound on the perimeter $|\pa\Om|$ only, by fixing $R=1$ (which always holds true up to a dilation), in the spirit of the main result in~\cite{CirMag}.

We notice that the $L^{N-2}$-deviation considered in (ii) of the above corollary is still weaker than the deviations considered in \cite{CirMag,JN}: in fact, \cite{CirMag} considered $C^0$-deviations and \cite{JN} considered $L^{N-1}$-deviations.

\bigskip

	
	In Theorem \ref{thm:stability with diameter} (and hence Corollary \ref{cor:two corollaries}), we obtained closeness of $\Om$ to a family of $m$ disjoint balls that are contained in $\Om$ and have radii that are quantitatively close to $R$ as specified by \eqref{eq:thm_new global statement radii}. By leveraging an iterative argument from \cite{CirMag}, it is not difficult to show that the conclusion of Theorem \ref{thm:stability with diameter} also implies the following alternative statement, which provides closeness to a family of $m$ disjoint balls with radii equal to $R$ (but not necessarily contained in $\Om$).  	
	\begin{thm}\label{thm:NewAlternative_equal radii}
		Let $\Om \subset \RR^N$, $N\ge 2$, be a bounded open set of class $C^2$ such that \eqref{eq:assumption lower bound M_0-} holds.
		Let $R$ and $H_0$ be those in \eqref{eq:def H_0 and R}, assume that $R\ge1$, and consider the deviation $\de$ defined in \eqref{eq:def delta}. 
		
		Then, there exist $m$ points $\widehat{z}_i$ such that the balls $B_{R} (\widehat{z}_i)$ (centered at $\widehat{z}_i$, with radius $R$), $i = 1, \dots, m$, are disjoint and the set 
		\begin{equation*}
			\widehat{\mathcal{F}}:= \bigcup_{i=1}^m B_{R} (\widehat{z}_i) 
		\end{equation*}
		satisfies
		\begin{equation}\label{eq:NewAlternative_thm_new global statement measure}
			\left| \Om \De \widehat{\mathcal{F}} \right| \le C \,  \de^{ \frac{\al}{ \max \left\lbrace 4 , N \right\rbrace }} ,
		\end{equation}
		\begin{equation}\label{eq:NewAlternative_thm_one-sided Hausdorff closeness}
			\max_{x\in \pa \widehat{\mathcal{F}} } \mathrm{dist}_{\pa\Om} (x) \le C \, \de^{\frac{\al}{4}},
		\end{equation}
		\begin{equation}\label{eq:NewAlternative_thm_perimeter closeness}
			\left| | \pa \Om| - | \pa \widehat{\mathcal{F}} | \right| \le C \,  \de^{ \frac{\al}{ \max \left\lbrace 4 , N \right\rbrace }} ;
		\end{equation}
		moreover, the number $m$ is bounded from above as in \eqref{eq:bound on number m}.
		The constants $C$ appearing in \eqref{eq:NewAlternative_thm_new global statement measure}, \eqref{eq:NewAlternative_thm_one-sided Hausdorff closeness}, \eqref{eq:NewAlternative_thm_perimeter closeness}, and \eqref{eq:bound on number m} (can be explicitly computed and) only depend on the dimension $N$, and upper bounds on the diameter $d_\Om$ of $\Om$ and $\cM_0^-$.
	\end{thm}
	
	\begin{rem}
		{\rm It is clear that an analogue of Corollary \ref{cor:two corollaries} also holds true for Theorem \ref{thm:NewAlternative_equal radii}.}
	\end{rem}


We stress that our results provide three different types of quantitative closeness: the first is an $L^1$-type estimate (such as \eqref{eq:thm_new global statement measure}, \eqref{eq:stab in L^N-2 II}, and \eqref{eq:NewAlternative_thm_new global statement measure}), the second is a one-sided Hausdorff-type estimate (such as \eqref{eq:thm_one-sided Hausdorff closeness}, \eqref{eq:stab in L^N-2_IV_one-sided Hausdorff closeness}, and \eqref{eq:NewAlternative_thm_one-sided Hausdorff closeness}), and the third is a perimeter bound (such as \eqref{eq:thm_perimeter closeness}, \eqref{eq:stab in L^N-2_V_perimeter closeness}, and \eqref{eq:NewAlternative_thm_perimeter closeness}). We finally recall that when the weak deviations (i.e., $L^1$-type or $L^{N-2}$-type deviations) considered in the present paper are replaced by
%
a stronger $C^0$-type deviation, then one can obtain
Hausdorff-type closeness to a set of disjoint balls that are mutually tangent with $C^{1,\al}$-closeness away from the tangency points \cite{CirMag} (by exploiting Allard's regularity theorem and a calibration-type argument).

\bigskip

The rest of the paper is organized as follows. Section \ref{sec:Preliminaries} contains preliminary results that will play a crucial role in the proof of Theorem \ref{thm:stability with diameter}. These include explicit bounds for the so-called torsion function and its gradient, as well as a new explicit quantitative upper bound for the volume of the interior tubular neighbourhood of $\pa\Om$, which may also be of independent interest. Section \ref{sec:proof} is devoted to the proof of Theorem \ref{thm:stability with diameter}, Corollary \ref{cor:two corollaries}, and Theorem \ref{thm:NewAlternative_equal radii}.

\section{Preliminaries}\label{sec:Preliminaries}

Given a $C^2$ domain $\Om\subset \RR^N$, we consider its so-called {\it torsion function} $u$, that is the solution of 
\begin{equation}\label{eq:torsionproblem}
	\begin{cases}
	\De u = N \qquad & \text{in } \Om ,
	\\
    u=0 \qquad & \text{on } \pa\Om .
    \end{cases}
\end{equation}
Let $H$ denote the mean curvature of $\pa \Om$, that is,
\begin{equation*}
H(x):=\frac{k_1(x) + \dots + k_{N-1}(x)}{N-1} , \quad \text{ for } x\in \pa \Om ,	
\end{equation*}
where $k_1(x), \dots , k_{N-1}(x)$ denote the $N-1$ principal curvatures of $\pa\Om$ at $x$.

Our starting point is the fundamental identity for Alexandrov's Soap Bubble Theorem that was obtained in \cite{MP} by leveraging the arguments used in \cite{Reilly}:
\begin{equation}\label{eq:SBT fundamental identity MP}
\frac{1}{N-1} \int_{\Om} \left\{ |D^2 u|^2 - \frac{(\De u)^2}{N} \right\} dx + \frac{1}{R} \int_{\pa\Om} \left(u_\nu - R\right)^2 dS_x = \int_{\pa\Om} \left( H_0 -H \right) u_\nu^2 \, dS_x ,
\end{equation}
where $R$ and $H_0$ are the reference constants defined in \eqref{eq:def H_0 and R}.
(See also \cite{FogagnoloPinamonti} for a recent generalization of such an identity to substatic manifolds with boundary of horizon type.)

From the above integral identity we immediately deduce the following.

\begin{lem}\label{lem:startinglemma}
Let $\Om\subset \RR^N$, $N\ge2$, be a bounded domain of class $C^2$. Then we have that
\begin{equation}\label{eq:CS deficit and positive part deviation}
\int_{\Om} \left\{ |D^2 u|^2 - \frac{(\De u)^2}{N} \right\} dx \le (N-1) G^2 \int_{\pa\Om} \left(H_0-H\right)^+ dS_x
\end{equation}
and
\begin{equation}\label{eq:Serrin deficit and positive part deviation}
\frac{1}{R} \int_{\pa\Om} \left(u_\nu - R\right)^2 dS_x	\le G^2 \, \int_{\pa\Om} \left(H_0-H\right)^+ dS_x ,
\end{equation}
where 
\begin{equation*}
G:=\max_{\pa \Om} |\na u|=\max_{\ol{\Om}} |\na u|.
\end{equation*}
\end{lem}

We will also need the following bounds for $u$ and its gradient. The gradient bound will be deduced from \cite[Lemma 2.2]{MP5} (that is the generalization to arbitrary dimension of a result established by Payne and Philippin \cite{PP} in dimension $2$).

\begin{lem}
Let $\Om \subset \RR^N$, $N\ge 2$, be a bounded domain of class $C^2$ and let $u$ be the solution of \eqref{eq:torsionproblem}. Then, we have that
\begin{equation}\label{eq:maximum estimate for u}
0 \le -u(x)  \le \frac{d_\Om^2}{2} \quad \text{ in } \ol{\Om}.
\end{equation}

If $\cM_0^-$ is that defined in \eqref{eq:assumption lower bound M_0-}, we have that
\begin{equation}\label{eq:gradient bound}
	G:= \max_{\ol{\Om}} | \na u| \le C(N, \cM_0^- , d_\Om) ,
\end{equation}
for an explicit constant $C$ only depending on $N, \cM_0^- , d_\Om$.

If $\pa\Om$ is mean convex, the dependence on $\cM_0^-$ can be removed.
\end{lem}
\begin{proof}
The inequalities in \eqref{eq:maximum estimate for u} easily follow by the maximum principle and comparing $u$ to the torsion function of a ball of radius $d_\Om$ containing $\Om$.
We mention that a finer upper bound for $-u$ may be obtained leveraging a classical result on rearrangements due to Talenti \cite{Talenti}; in fact, exploiting \cite{Talenti} we can deduce that
$$\max_{\ol{\Om}}(-u) \le \frac{\left( \frac{|\Om|}{|B_1|} \right)^{\frac{2}{N}} - |x|^2 }{2} \qquad \text{ for any } x \in B_1:=\left\lbrace x\in\RR^N \, : \, |x|<1 \right\rbrace ,$$
and hence
\begin{equation}\label{eq:finer bound maximum}
\max_{\ol{\Om}}(-u) \le \frac{1}{2} \left( \frac{|\Om|}{|B_1|} \right)^{\frac{2}{N}} .
\end{equation}
The bound in \eqref{eq:maximum estimate for u} (which clearly also follows from the last finer bound) is enough for our purposes.

By \cite[Lemma 2.2]{MP5} we have that
\begin{equation*}
	G^2 \le 2(N-1) \cM_0^- G \max_{\ol{\Om}}(-u) + 2N \max_{\ol{\Om}}(-u) ,
\end{equation*}
and hence,
\begin{equation*}
	G \le  (N-1) \cM_0^- \max_{\ol{\Om}}(-u) + \sqrt{ \left[ (N-1) \cM_0^- \max_{\ol{\Om}}(-u)\right]^2 + 2N \max_{\ol{\Om}}(-u)  } ,
\end{equation*}
from which \eqref{eq:gradient bound} follows recalling the second inequality in \eqref{eq:maximum estimate for u}. It is clear that using \eqref{eq:finer bound maximum} instead of \eqref{eq:maximum estimate for u} would give a finer explicit bound for $G$ in terms of $N$, $\cM_0^-$, and $|\Om|$ only; we do not need such a finer bound in the present paper. 
\end{proof}

We finally need a quantitative explicit upper bound on the measure of the tubular neighbourhood of $\pa\Om$. We recall that the classical Steiner-Weyl formula \cite{Weyl} provides a polynomial-type formula for such a measure, which was extended by Federer to sets of positive reach in \cite{Fe}; such a formula, whose coefficients depend on second order properties of $\pa \Om$, holds true provided that $\eta$ is small enough (in terms of the so-called {\it reach} of $\Om$, see \cite{Fe}) and, in particular, gives the well-known asymptotic behaviour:
\begin{equation}\label{eq:asymptotic tubular}
\lim_{\eta \to 0^+} \frac{\left| \left\lbrace x\in \Om \, : \, \mathrm{dist}_{\pa\Om} (x) \le \eta \right\rbrace \right|}{\eta} = \left| \pa\Om \right|.
\end{equation}
However, in order to establish our main results, we will need the following possibly new quantitative explicit upper bound for $\left| \left\lbrace x\in \Om \, : \, \mathrm{dist}_{\pa\Om} (x) \le \eta \right\rbrace \right|$, which we prove below.

\begin{lem}\label{lem:upper bound tube}
Let $\Om \subset \RR^N$, $N\ge 2$, be a bounded domain of class $C^2$ and let $\cM_0^-$ be that defined in \eqref{eq:assumption lower bound M_0-}. Then, we have that
\begin{equation}\label{eq:upper bound tubular neighbourhood}
	\left| \left\lbrace x\in \Om \, : \, \mathrm{dist}_{\pa\Om} (x) \le \eta \right\rbrace \right| \le \left( 1 + \eta \cM_0^- \right)^{N-1} \, \left| \pa\Om \right| \, \eta . 
\end{equation}
\end{lem}

It is clear that \eqref{eq:upper bound tubular neighbourhood} recovers \eqref{eq:asymptotic tubular}, but also gives an explicit quantitative information in terms of $\cM_0^-$ and $|\pa\Om|$ only, which is necessary to prove our main results. Notice that no smallness assumption on $\eta$ is required in \eqref{eq:upper bound tubular neighbourhood}.

\begin{proof}[Proof of Lemma \ref{lem:upper bound tube}]
Denoting with $k_1(x) \le \dots \le k_{N-1}(x)$ the principal curvatures of $\pa\Om$ at $x\in \pa\Om$, we consider the following partition of $\pa\Om$:
\begin{equation*}
\Ga_1:= \left\lbrace x \in \pa \Om \, : \, k_{N-1}(x) \le \frac{1}{\eta} \right\rbrace	
\quad \text{and} \quad
	\Ga_2:= \left\lbrace x \in \pa \Om \, : \, k_{N-1}(x) > \frac{1}{\eta} \right\rbrace .	
\end{equation*}
We now show that
\begin{equation*}
	\left\lbrace x \in \Om \, : \, \mathrm{dist}_{\pa\Om} (x) \le \eta \right\rbrace \subset
	\left\lbrace x - t \nu(x) \, : \, x \in \Ga_1 , \, 0 \le t \le \eta \right\rbrace 
	\cup \left\lbrace x - t \nu(x) \, : \, x \in \Ga_2 , \, 0 \le t \le \frac{1}{k_{N-1}(x)} \right\rbrace ,
\end{equation*}	
where as usual $\nu$ denotes the exterior unit normal to $\pa\Om$.
For any given $y \in \left\lbrace x \in \Om \, : \, \mathrm{dist}_{\pa\Om} (x) \le \eta \right\rbrace $, consider $x \in \pa \Om$ such that $\mathrm{dist}_{\pa\Om}(y) = |x-y|$.
It is clear that $\mathrm{dist}_{\pa\Om}(y) \le \eta$; moreover, since $k_{N-1} (x) \le 1/ \mathrm{dist}_{\pa\Om}(y)$, if $x \in \Ga_2$ then we also have that
$\mathrm{dist}_{\pa\Om}(y) \le 1/ k_{N-1} (x) $, which proves the desired inclusion.

Next, we compute
\begin{equation*}
	\left| \left\lbrace x \in \Om \, : \, \mathrm{dist}_{\pa\Om} (x) \le \eta \right\rbrace  \right| 
	\le \int_{\Ga_1} \int_0^\eta \prod_{i=1}^{N-1} \left| 1 - t \,  k_i (x) \right| dt \, dS_x  + \int_{\Ga_2} \int_0^{ 1/ k_{N-1} (x) } \prod_{i=1}^{N-1} \left| 1 - t \, k_i (x) \right| dt \, dS_x .
\end{equation*}
By definition of $\Ga_1$ and $\Ga_2$, we have that, for any $i=1, \dots, N-1$,
\begin{equation*}
	\left| 1 - t \, k_i (x) \right| =  1 - t \,  k_i (x) \quad \text{ for any }
	(x,t) \in \left( \Ga_1 \times \left[ 0 , \eta \right] \right) \cup \left( \Ga_2 \times \left[ 0, 1 / k_{N-1}(x)  \right] \right) ,
\end{equation*}
and hence, in both the integrals above, we have that
\begin{equation*}
\begin{split}
\prod_{i=1}^{N-1} \left| 1 - t \, k_i (x) \right| 
& = \prod_{i=1}^{N-1} \left( 1 - t \,  k_i (x) \right)
\\
& \le \left( \frac{1}{N-1} \sum_{i=1}^{N-1} \left( 1 - t \,  k_i (x) \right) \right)^{N-1}
\\
& = \left( 1 - t H(x) \right)^{N-1}
\\
& \le \left( 1 + t \, \cM_0^- \right)^{N-1}
\end{split}
\end{equation*}
where we used the arithmetic-geometric inequality and \eqref{eq:assumption lower bound M_0-}. Hence, we get that
\begin{equation*}
	\begin{split}
	\left| \left\lbrace x \in \Om \, : \, \mathrm{dist}_{\pa\Om} (x) \le \eta \right\rbrace  \right| 
	& \le \int_{\Ga_1} \int_0^\eta \left( 1 + t \cM_0^- \right)^{N-1} \, dt \, dS_x  + \int_{\Ga_2} \int_0^{ 1/ k_{N-1} (x) } \left( 1 + t \cM_0^- \right)^{N-1} \, dt \, dS_x 
	\\
	& \le \int_{\Ga_1} \int_0^\eta \left( 1 + t \cM_0^- \right)^{N-1} \, dt \, dS_x  + \int_{\Ga_2} \int_0^{ \eta } \left( 1 + t \cM_0^- \right)^{N-1} \, dt \, dS_x ,
	\end{split}
\end{equation*}
where we used that, by definition of $\Ga_2$, $1/ k_{N-1} (x) < \eta$ on $\Ga_2$.
Thus, recalling that $\Ga_1 \cup \Ga_2 = \pa \Om$, we find
\begin{equation}
\left| \left\lbrace x \in \Om \, : \, \mathrm{dist}_{\pa\Om} (x) \le \eta \right\rbrace  \right| \le 
\int_{\Ga_1 \cup \Ga_2} \int_0^\eta \left( 1 + t \cM_0^- \right)^{N-1} \, dt \, dS_x \le \left( 1 + \eta  \cM_0^- \right)^{N-1} \, |\pa \Om| \, \eta ,	
\end{equation}
that is the desired result.
\end{proof}

\section{Proofs of the main results}\label{sec:proof}

Before proceeding with the proof of Theorem \ref{thm:stability with diameter}, we start with the following two comments.

\begin{rem}\label{rem:Initial remarks on proof of THM}
{\rm 
\begin{enumerate}[(i)]
	\item We stress that Theorem \ref{thm:stability with diameter} is a global result that does not require any smallness assumption on the deviation $\de$. We will prove Theorem \ref{thm:stability with diameter} when $\de \le c$, where $c$ is some explicit constant only depending on the parameters indicated in the statement, i.e., $N$, $d_\Om$, and $\cM_0^-$. On the other hand, if $\de \ge c$,
	then 
the conclusion of Theorem \ref{thm:stability with diameter} trivially holds true,
being as
	\begin{equation*}
		 R \le \left( \frac{|\Om|}{|B_1|} \right)^{\frac{1}{N}} \le d_\Om \le \frac{d_\Om}{c^{\frac{\al}{4}}} \,  \de^{\frac{\al}{4}}  ,
	\end{equation*}
	\begin{equation*}
		\left| \Om \right| \le |B_1| \, d_\Om^N \le \frac{|B_1| \, d_\Om^N}{c^{ \frac{\al}{ \max \left\lbrace 4 , N \right\rbrace }}}  \de^{ \frac{\al}{ \max \left\lbrace 4 , N \right\rbrace }} ,
	\end{equation*}
\begin{equation*}
		\sup_{x\in \Om } \mathrm{dist}_{\pa\Om} (x) \le d_\Om \le \frac{d_\Om}{c^{\frac{\al}{2}}} \,  \de^{\frac{\al}{2}}  ,
\end{equation*}
\begin{equation*}
	| \pa\Om| = \frac{N | \Om |}{R} \le N| \Om| \le N \frac{|B_1| \, d_\Om^N}{c^{ \frac{\al}{ \max \left\lbrace 4 , N \right\rbrace }}}  \de^{ \frac{\al}{ \max \left\lbrace 4 , N \right\rbrace }}.
\end{equation*}
Here, in the first inequality we used the isoperimetric inequality, whereas in the last line we used the definition of $R$ in \eqref{eq:def H_0 and R} and the assumption $R \ge 1$.	
	\item In the proof of Theorem \ref{thm:stability with diameter} we will keep track of the constants involved, providing explicit estimates in terms of the parameters indicated in the statement, i.e., $N$, $d_\Om$, and $\cM_0^-$; in some of the constants we will make use of $|\Om|$ and $G$, but it is clear that they can be explicitly estimated by means of $|\Om| \le |B_1| d_\Om^N$ and \eqref{eq:gradient bound}.	
\end{enumerate}
}
\end{rem}

We are now ready to proceed with the

\begin{proof}[Proof of Theorem \ref{thm:stability with diameter}]
Let $\de$ be that defined in \eqref{eq:def delta} and assume that
\begin{equation}\label{eq:delta small minore di 1}
\de \le 1 .
\end{equation}
Let $u$ be the solution of \eqref{eq:torsionproblem}.

{\bf Step 1.}
We set $\al \in ( 0 , 1)$ as in
\eqref{eq:alpha}.
From \eqref{eq:CS deficit and positive part deviation} we
easily deduce (e.g., by Chebyshev's inequality)
that there exists a set $A_\de$ such that
\begin{equation}\label{eq:set Adelta}
	|A_\de| \le \de^{1-\al}
\end{equation}
and
\begin{equation}\label{eq: CS pointwise in Om -Ade}
\quad |D^2 u|^2 - \frac{(\De u)^2}{N} \le (N-1) G^2 \de^\al \text{ in } \Om \setminus A_\de ,
\end{equation}
where, as in \eqref{eq:gradient bound}, $G:= \max_{\ol{\Om}} |\na u|$.

For $0 < \ve < \max_{\ol{\Om}} (-u)$, set 
\begin{equation*}
\Om_\ve:= \left\lbrace x \in \Om \, : \, u < - \ve \right\rbrace .
\end{equation*}

It is clear that
	\begin{equation*}
	|\Om_\ve| \ge	|B_1|\left( \frac{\max_{\ol{\Om}}(-u) - \ve }{G} \right)^N  .
\end{equation*}

Notice that, being as $u=0$ on $\pa\Om$, we have that
\begin{equation}\label{eq:distance lower bound epsilon}
\mathrm{dist}_{\pa\Om} (x) \ge \frac{\ve}{G}	 \quad \text{ for any } x\in \Om_\ve ,
\end{equation}
where $\mathrm{dist}_{\pa\Om} (x)$ denotes the distance of $x$ to $\pa\Om$.

By classical interior regularity estimates (\cite{GT}), there exists a constant $\ol{C}(N)$ (only depending on $N$) such that
\begin{equation}\label{eq:interior estimates hessian}
|D^2 u (x)| \le \frac{\ol{C}(N)}{\mathrm{dist}_{\pa\Om} (x)^2 } \le \frac{\ol{C}(N) \, G^2}{\ve^2} \quad \text{ for any } x\in \Om_\ve ,
\end{equation}
where in the second inequality we used \eqref{eq:distance lower bound epsilon}.

Define
\begin{equation}\label{eq:def h}
h(x) := q(x) - u(x) \, \text{ for } x\in\ol{\Om}, \quad \text{where } q(x):=\frac{|x|^2}{2} \, \text{ for } x \in \RR^N ,
\end{equation}
and notice that, since $\De u =N$ in $\Om$, $h$ is harmonic in $\Om$ and
\begin{equation}\label{eq:relation Hessian h CS}
|D^2 h|^2 = |D^2 u|^2 - \frac{(\De u)^2}{N} .
\end{equation}

For any $x\in \Om_{2\ve}$ and $i,j=1\dots,N$, we compute that
\begin{equation}\label{eq:prima catena di disuguaglianze}
\begin{split}
\left| \frac{\pa^2}{\pa x_i \pa x_j}h (x) \right| 
& = \left| \frac{1}{|B_{\frac{\ve}{G}}|} \int_{B_{\frac{\ve}{G}}(x)} \frac{\pa^2}{\pa x_i \pa x_j}h (y) dy \right| 
\\
& = \frac{1}{|B_{\frac{\ve}{G}}|} \left|  \int_{B_{\frac{\ve}{G}}(x) \setminus A_\de} \frac{\pa^2}{\pa x_i \pa x_j}h (y) dy + \int_{B_{\frac{\ve}{G}}(x) \cap A_\de} \frac{\pa^2}{\pa x_i \pa x_j} h (y) dy \right|
\\
& \le
\frac{|B_{\frac{\ve}{G}}(x) \setminus A_\de|}{|B_{\frac{\ve}{G}}|}  \nr D^2 h \nr_{L^{\infty} ( B_{\frac{\ve}{G}}(x) \setminus A_\de )} + \frac{| B_{\frac{\ve}{G}}(x) \cap A_\de |}{|B_{\frac{\ve}{G}}|} \nr D^2 h \nr_{L^{\infty} ( B_{\frac{\ve}{G}}(x) \cap A_\de ) }
\\
& \le \sqrt{N-1} \,  G \,  \de^{\frac{\al}{2}}  + \frac{\de^{1-\al} G^N}{|B_1| \ve^{N}} \left(\sqrt{N} + \frac{\ol{C}(N)G^2}{\ve^2} \right)
\end{split}
\end{equation}
where in the last inequality we used that
\begin{equation*}
	\frac{|B_{\frac{\ve}{G}}(x) \setminus A_\de|}{|B_{\frac{\ve}{G}}|} \le 1 ,
\end{equation*}
\begin{equation*}
\nr D^2 h \nr_{L^{\infty} ( B_{\frac{\ve}{G}} (x) \setminus A_\de )} \le \sqrt{N-1} \,  G \,  \de^{\frac{\al}{2}} 
\qquad \text{(by \eqref{eq: CS pointwise in Om -Ade} and \eqref{eq:relation Hessian h CS})} ,
\end{equation*}
\begin{equation*}
| B_{\frac{\ve}{G}}(x) \cap A_\de | \le |A_\de|\le \de^{1-\al} \qquad \text{(by \eqref{eq:set Adelta})} ,
\end{equation*}
and
\begin{equation*}
\nr D^2 h \nr_{L^{\infty} ( B_{\frac{\ve}{G}} (x) \cap A_\de )} = \nr I - D^2 u \nr_{L^{\infty} ( B_{\frac{\ve}{G}} (x) \cap A_\de )} \le  \nr I - D^2 u \nr_{L^{\infty} \left( \Om_\ve \right)} \le \sqrt{N} + \frac{\ol{C}(N) G^2}{\ve^2} ,
\end{equation*}
being as $B_{\frac{\ve}{G}}(x) \cap A_\de \subset \Om_\ve$ for $x\in\Om_{2\ve}$ and recalling \eqref{eq:interior estimates hessian}. 
Setting 
\begin{equation}\label{eq:def vepsilon}
	\ve:=\de^\al \quad \text{ with } \al \text{ as in \eqref{eq:alpha} }
\end{equation}
and recalling \eqref{eq:delta small minore di 1}, we clearly have that
\begin{equation*}
\frac{\de^{1-\al} G^N }{|B_1| \ve^{N}} \left( \sqrt{N} + \frac{\ol{C}(N)G^2}{\ve^2} \right) = \frac{\de^{1 - (N+1) \al } G^N }{|B_1| }  \left(\sqrt{N} + \frac{\ol{C}(N)G^2}{\de^{2 \al }} \right) \le C_1 \, \de^{1 - ( N + 3 ) \al} , 
\end{equation*}
where we have set
\begin{equation}\label{eq:C_1}
	C_1 :=  \frac{ G^N  }{|B_1|} \left( \sqrt{N} + \ol{C}(N) \, G^2 \right) ;
\end{equation}
hence,
\eqref{eq:prima catena di disuguaglianze} gives that, for any $x\in \Om_{2\ve}$ and $i,j=1\dots,N$,
\begin{equation*}
	\left| \frac{\pa^2}{\pa x_i \pa x_j}h (x) \right|  \le \sqrt{N-1} \,  G \,  \de^{\frac{\al}{2}}  + C_1 \, \de^{1 - ( N + 3 ) \al} = \left( \sqrt{N-1} \,  G + C_1 \right) \de^{\frac{\al}{2}} ,
\end{equation*}
where in the last equality we used that, by \eqref{eq:alpha}, $1 - ( N + 3 ) \al = \al/2$.
This clearly gives that
\begin{equation}\label{eq:Hessian h estimate}
	\nr D^2 h \nr_{L^\infty(\Om_{2 \ve})} \le C_2 \, \de^{\frac{\al}{2}} ,
\end{equation}
where
\begin{equation}\label{eq:C_2}
	C_2:= N \left( \sqrt{N-1} \,  G + C_1 \right) \quad \text{with } C_1 \text{ as in \eqref{eq:C_1}}.
\end{equation}
As a consequence, we have that if
\begin{equation}\label{eq:smallness delta 2}
	\de < \frac{1}{C_2^{\frac{2}{\al}}} ,
\end{equation}
then $D^2 u = I - D^2 h$ is positive definite in $\Om_{2\ve}=\Om_{2 \de^\al}$; therefore,
denoting with\footnote{Here, $\cI$ is an at most countable set.} $\left\lbrace \Om^i_{2 \ve} \right\rbrace_{ i \in \cI   }$ the connected components of $\Om_{2 \ve}$, we have that, for any $i \in \cI $, $\Om_{2 \ve}^i$ is convex (provided that \eqref{eq:smallness delta 2} holds true). In this case, in each $\ol{\Om}^i_{2 \ve}$, $u$ has only one minimum point (see, e.g., \cite{Ko}), which clearly belongs to $\Om_{2 \ve}^i$.

{\bf Step 2.} For $i \in \cI $, we denote with $z_i \in \Om_{2 \ve}^i$ the (unique) minimum point of $u$ in $\ol{\Om}_{2 \ve}^i$, and we define
\begin{equation}\label{eq:def h_z_i}
	h_{z_i}(x):= q_{z_i}(x) - u(x) \, \text{ for } x \in \ol{\Om}^i_{2\ve} ,
	\quad \text{ where }
	q_{z_i}(x):=\frac{|x-z_i|^2}{2} + u(z_i) \, \text{ for } x \in \RR^N . 
\end{equation}
It is readily checked that, for any $i \in \cI $, we have that
\begin{equation*}
	D^2 h_{z_i} = D^2 h,
\end{equation*}
where $h$ is that defined in \eqref{eq:def h}, and hence \eqref{eq:Hessian h estimate} gives that
\begin{equation}\label{eq:Hessian estimate h_z_i}
 \nr D^2 h_{z_i} \nr_{L^\infty(\Om^i_{2 \ve})} \le C_2 \, \de^{\frac{\al}{2}} \quad \text{for any } i \in \cI .
\end{equation}
Moreover, by definition \eqref{eq:def h_z_i} we also readily check that
\begin{equation*}
	h_{z_i}(z_i)=0 \quad \text{ and } \quad \na h_{z_i} (z_i)=0 .
\end{equation*}
%
%
Since $\Om^i_{2 \ve}$ is convex, Taylor expansions give that\footnote{We mention that \eqref{eq:Taylor for the gradient of h_{z_i}} remains valid even without $\sqrt{N}$, but since we are not concerned with optimal constants, we do not need this refinement here.}
\begin{equation}\label{eq:Taylor for h_{z_i}}
\nr h_{z_i} \nr_{L^\infty(\Om^i_{2 \ve})} \le \nr D^2 h_{z_i} \nr_{L^\infty(\Om^i_{2 \ve})} \frac{(\rho^i_{ext})^2}{2} \le \frac{(\rho^i_{ext})^2}{2} \, C_2 \, \de^{\frac{\al}{2}} ,
\end{equation}
\begin{equation}\label{eq:Taylor for the gradient of h_{z_i}}
	\nr \na h_{z_i} \nr_{L^\infty(\Om^i_{2 \ve})} \le \nr D^2 h_{z_i} \nr_{L^\infty(\Om^i_{2 \ve})} (\rho^i_{ext}) \, \sqrt{N} \le (\rho^i_{ext}) \, \sqrt{N} \, C_2 \, \de^{\frac{\al}{2}} ,
\end{equation}
where we used \eqref{eq:Hessian estimate h_z_i} and set
\begin{equation}\label{eq:def rho_ext^i}
\rho^i_{ext}:= \max_{\pa \Om^i_{2\ve}} | x-z_i | .	
\end{equation}
We also define
\begin{equation}\label{eq:def rho_int^i}
	\rho^i_{int}:= \min_{\pa \Om^i_{2\ve}} | x-z_i | ;	
\end{equation}
Clearly, by definition we have that
\begin{equation}\label{eq:inclusions Brhoint Om2ve Brhoext and rhointlerhoext}
\rho^i_{int} \le \rho^i_{ext} \qquad \text{ and } \qquad 
B_{\rho^i_{int}}(z_i) \subset \Om_{2\ve}^i \subset B_{\rho^i_{ext}}(z_i) ;
\end{equation}
we now show that
\begin{equation}\label{eq:rho_ext-rho_int}
\rho^i_{ext} - \rho^i_{int} \le  2 \, d_\Om \, C_2 \, \de^{\frac{\al}{2}} \qquad \text{for any } i \in \cI .
\end{equation}

To prove \eqref{eq:rho_ext-rho_int}, we use the definitions \eqref{eq:def h_z_i}, \eqref{eq:def rho_ext^i}, and \eqref{eq:def rho_int^i} to easily compute that
\begin{equation*}
	\begin{split}
	\rho^i_{ext} - \rho^i_{int} 
	& \le \frac{\left(\rho^i_{ext} + \rho^i_{int}\right)}{\rho^i_{ext}} \left( \rho^i_{ext} - \rho^i_{int} \right)
	\\
	& =
	\frac{2}{\rho^i_{ext}} \left\lbrace \max_{x\in \pa \Om^i_{2\ve}} q_{z_i}(x) - \min_{x\in \pa \Om^i_{2\ve}} q_{z_i}(x) \right\rbrace
	\\
	& = \frac{2}{\rho^i_{ext}} \left\lbrace \max_{x\in \pa \Om_{2\ve}^i} h_{z_i}(x) - \min_{x\in \pa \Om^i_{2\ve}} h_{z_i}(x) \right\rbrace
	\\
	& \le \frac{4}{\rho^i_{ext}} \nr h_{z_i} \nr_{L^\infty(\Om^i_{2 \ve})} 
	\\
	& \le 2 \, \rho^i_{ext} \, C_2 \, \de^{\frac{\al}{2}} .
	\end{split}
\end{equation*}
Here, in the last equality we used that, by definition of $\Om_{2\ve}$, $u$ is constant on $\pa \Om_{2\ve}$ (more precisely, $u \equiv -2\ve$ on $\pa \Om_{2\ve}$), 
and the last inequality follows from \eqref{eq:Taylor for h_{z_i}}.
Hence, by using the trivial inequality
\begin{equation}\label{eq:trivial inequality rhoext and dOm}
	\rho_{ext}^i \le d_\Om \qquad \text{for any } i \in \cI ,
\end{equation}
\eqref{eq:rho_ext-rho_int} easily follows.

By \eqref{eq:inclusions Brhoint Om2ve Brhoext and rhointlerhoext}, it is clear in particular that
\begin{equation}\label{eq:star per comodit alla fine}
	\left\lbrace B_{\rho_{int}^i} (z_i) \right\rbrace_{ i \in \cI } \, \text{ are disjoint balls and }
    \, \bigcup_{i \in \cI} B_{\rho_{int}^i}(z_i) \subset \Om_{2 \ve } \subset \Om .
\end{equation}

Recalling that, by the maximum principle we have that (see, e.g., \cite[Lemma 3.1]{MP2})
	\begin{equation*}
		\frac{\mathrm{dist}_{\pa\Om}(x)^2}{2} \le -u(x)	\quad \text{ for any } x\in \ol{\Om} ,
	\end{equation*}
	and using \eqref{eq:rho_ext-rho_int}, we observe that
	\begin{equation}\label{eq:provadarichiamare perindicare minimization}
		\mathrm{dist}_{\pa\Om}(x) \le 2  \ve^{\frac{1}{2}} + 2 \, d_\Om \, C_2 \, \de^{\frac{\al}{2}}  \quad \text{ for any } x\in \Om \setminus \bigcup_{i \in \cI} B_{\rho_{int}^i}(z_i),
	\end{equation}
	that is, recalling \eqref{eq:def vepsilon}, 
	\begin{equation}\label{eq:inclusion tubular}
		\Om \setminus \bigcup_{i \in \cI} B_{\rho_{int}^i}(z_i) \subset	\left\lbrace x \in \Om \, : \, \mathrm{dist}_{\pa\Om}(x) \le \left( 2 + 2 \, d_\Om \, C_2 \right) \, \de^{\frac{\al}{2}} \right\rbrace .
	\end{equation}
	Combining \eqref{eq:inclusion tubular} and Lemma \ref{lem:upper bound tube}, we thus obtain that
\begin{equation}\label{eq:stima misura complementare di tutte le palle}
	\left| \Om \setminus \bigcup_{i \in \cI} B_{\rho_{int}^i}(z_i) \right| \le \left\lbrace \left[ 1 + \left( 2 + 2 \, d_\Om \, C_2 \right) \cM_0^- \right]^{N-1}  \left( 2 + 2 \, d_\Om \, C_2 \right)   \right\rbrace \, | \pa\Om | \, \de^{\frac{\al}{2}} ,
\end{equation}
where we also made use of \eqref{eq:delta small minore di 1}.

In the remaining steps we shall prove that, actually, only a finite number of $B_{\rho_{int}^i}(z_i)$ have radius ``close'' to $R$ and these balls are ``enough'' to approximate $\Om$, as the possibly remaining balls have ``negligible'' radii.

{\bf Step 3.}
We now prove that
\begin{equation}\label{eq:fundamental inequality 1}
\left|	\int_{\Om} (-u) dx - \frac{R^2 | \Om | }{N+2} \right| \le C_3 \, \de^{\frac{1}{2}} 
\qquad \text{with} \quad C_3:= 2 \, \left( N | \Om |\right)^{\frac{1}{2}} \, d_\Om \, G^2 ,
\end{equation}
where $G$ is that defined in \eqref{eq:gradient bound}.
To this aim, recalling that $u$ solves \eqref{eq:torsionproblem}, integrating by parts we see that
\begin{equation*}
	\int_{\Om} (-u) dx = \frac{1}{N} \int_{\Om} |\na u |^2 dx = \frac{1}{N(N+2)}  \int_{\pa \Om} u_\nu^2 \langle x - z , \nu  \rangle dS_x , 
\end{equation*}
where the last identity follows by the classical Pohozaev identity \cite{Poh} and holds true for any given $z\in\RR^N$ (see, e.g, the last identity in the proof of \cite[Theorem 2.3]{Pog}); here, we fix $z$ to be any point in $\Om$.
Subtracting $\frac{R^2 | \Om | }{N+2}$ from both sides, we thus have that
\begin{equation*}
	\int_{\Om} (-u) dx - \frac{R^2 | \Om | }{N+2} = \frac{1}{N(N+2)}  \int_{\pa \Om} (u_\nu^2 - R^2) \langle x - z , \nu  \rangle dS_x ,
\end{equation*}
from which \eqref{eq:fundamental inequality 1} follows by using that
\begin{equation}\label{eq:AA la richiamo per dopo senza diametro}
	\begin{split}
\left| 	\int_{\pa \Om} (u_\nu^2 - R^2) \langle x - z , \nu  \rangle dS_x \right|
& \le d_\Om \left( R + \max_{\pa\Om} |\na u| \right) \int_{\pa \Om} | u_\nu - R | dS_x
\\
& \le d_\Om \left( R + \max_{\pa\Om} |\na u| \right) | \pa \Om |^{\frac{1}{2}} \nr u_\nu - R \nr_{L^2 (\pa\Om)}
\\
& \le  d_\Om \left( R + \max_{\pa\Om} |\na u| \right) \left( N | \Om |\right)^{\frac{1}{2}} \, G \,  \de^{\frac{1}{2}}
\\
& \le 2 \, \left( N | \Om |\right)^{\frac{1}{2}} \, d_\Om \, G^2 \, \de^{\frac{1}{2}} ; 
    \end{split}
\end{equation}
here, in the first inequality we used that $\langle x - z , \nu  \rangle \le d_\Om$ (being as $z \in \Om$), the second inequality follows from the H\"{o}lder inequality, in the third inequality we used \eqref{eq:Serrin deficit and positive part deviation}, and in the last inequality we used that, by definitions of $R$ and $G$ (i.e., \eqref{eq:def H_0 and R} and \eqref{eq:gradient bound}),
\begin{equation}\label{eq:stima da sopra di R}
	R = \frac{1}{|\pa\Om|} \int_{\pa\Om} u_\nu dS_x \le \max_{\pa\Om} |\na u| = G .
\end{equation}

{\bf Step 4.}
We prove that
\begin{equation}\label{eq:fundamental inequality 2}
	\left| \int_{\Om} (-u) dx - \frac{ | B_1 | }{N+2} \sum_{i \in \cI} \left( \rho_{int}^i \right)^{N+2} \right| \le C_4 \, \de^{\frac{\al}{2}} ,
\end{equation}
where $| B_1 |$ denotes the volume of the unit ball in $\RR^N$.
To this aim, we notice that
\begin{equation}\label{eq:stimare integral -u starting eq}
\begin{split}
\int_{\Om} (-u) dx - \frac{ | B_1 | }{N+2} \sum_{i \in \cI} \left( \rho_{int}^i \right)^{N+2}
= &
\int_{\Om \setminus \bigcup_{i \in \cI} B_{\rho_{int}^i}(z_i) }	(-u) dx
\\
& +
\int_{\bigcup_{i \in \cI} B_{\rho_{int}^i}(z_i) }	h_{z_i} dx
\\
& +
\int_{\bigcup_{i \in \cI} B_{\rho_{int}^i}(z_i) }	(- q_{z_i}) dx - \frac{ | B_1 | }{N+2} \sum_{i \in \cI} \left( \rho_{int}^i \right)^{N+2} ,
\end{split}
\end{equation}
where $q_{z_i}$ and $h_{z_i}$ are those defined in \eqref{eq:def h_z_i}.

We start estimating the first summand on the right-hand side of \eqref{eq:stimare integral -u starting eq}; 
using that
\begin{equation*}
	- u(x) \le 2 \ve \qquad \text{ for any } x \in \Om \setminus \Om_{2 \ve}
\end{equation*}
and that, by \eqref{eq:gradient bound} and \eqref{eq:rho_ext-rho_int},
\begin{equation*}
	\begin{split}
	- u(x) & \le 2 \ve + G \left( \rho_{ext}^i - \rho_{int}^i \right)  
	\\
	& \le 2 \ve + 2 G \, d_\Om \, C_2 \, \de^{\frac{\al}{2}}    \qquad \text{ for any } x \in  \Om^i_{2 \ve}\setminus B_{\rho_{int}^i}(z_i)  \text{ and } i \in \cI ,
	\end{split}
\end{equation*}
we easily deduce that
\begin{equation}\label{eq:aggiunto step evitando tubular}
- u(x) \le 2  \ve + 2 G \, d_\Om \, C_2 \, \de^{\frac{\al}{2}} 	\le \left( 2  + 2 G \, d_\Om \, C_2 \right)\, \de^{\frac{\al}{2}} \quad \text{ for any } x \in \Om \setminus \bigcup_{i \in \cI} B_{\rho_{int}^i}(z_i) 	;
\end{equation}
in the last inequality we used \eqref{eq:def vepsilon} and \eqref{eq:delta small minore di 1}.
Using \eqref{eq:aggiunto step evitando tubular}, we thus compute that
\begin{equation}\label{eq:term I for step 3}
		\int_{\Om \setminus \bigcup_{i \in \cI} B_{\rho_{int}^i}(z_i) }	(-u) dx	
		\le |\Om|  \left( 2  + 2 G \, d_\Om \, C_2 \right)\, \de^{\frac{\al}{2}} .
\end{equation}
We mention that a finer estimate for this term could be obtained by exploiting \eqref{eq:stima misura complementare di tutte le palle} (instead of using the rough estimate $\left| \Om \setminus \bigcup_{i \in \cI} B_{\rho_{int}^i}(z_i) \right| \le |\Om|$), but we stress that such a refinement is not needed here, as the worst power rate would however appear in \eqref{eq:step5 term 3}.

%

For the second summand on the right-hand side of \eqref{eq:stimare integral -u starting eq}, recalling \eqref{eq:Taylor for h_{z_i}} and that $B_{\rho_{int}^i} (z_i) \subset \Om_{2\ve}^i$, we compute:
\begin{equation}\label{eq:term II for step 3}
	\left| \int_{\bigcup_{i \in \cI} B_{\rho_{int}^i}(z_i) }	h_{z_i} dx \right| \le \frac{ d_\Om^2}{2} C_2 |\Om| \,  \de^{\frac{\al}{2}};
\end{equation}
here, we also used the trivial inequality \eqref{eq:trivial inequality rhoext and dOm}.

For the last line of \eqref{eq:stimare integral -u starting eq}, noting that
\begin{equation*}
	\begin{split}
\int_{B_{\rho_{int}^i}(z_i)} \frac{ (\rho_{int}^i)^2 - |x - z_i|^2 }{2} dx
& =
\frac{(\rho_{int}^i)^2 }{2} | B_{\rho_{int}^i}(z_i)| - \frac{1}{2(N+2)}\int_{\pa B_{\rho_{int}^i}(z_i)} |x-z_i|^2 \langle x- z_i , \nu \rangle dS_x 
\\
& =
\frac{ | B_1 | \left( \rho_{int}^i \right)^{N+2} }{N+2} ,
\end{split}
\end{equation*}
we compute
\begin{equation}\label{eq:term IIIA for step 3}
	\begin{split}
	 \int_{\bigcup_{i \in \cI} B_{\rho_{int}^i}(z_i) }	(- q_{z_i}) dx - \frac{ | B_1 | }{N+2} \sum_{i \in \cI} \left( \rho_{int}^i \right)^{N+2} 
	& = \int_{\bigcup_{i \in \cI} B_{\rho_{int}^i}(z_i) }
	\left(  \frac{ |x - z_i|^2 -  (  \rho_{int}^i)^2  }{2} - q_{z_i}  \right) dx
	\\
	& =  \int_{\bigcup_{i \in \cI} B_{\rho_{int}^i}(z_i) } \left( - u(z_i) - \frac{(\rho_{int}^i)^2 }{2} \right) dx ;
	\end{split}
\end{equation}
in what follows we shall prove that
\begin{equation}\label{eq:comparison estimates u(zi) rho_i}
	0 \le -u(z_i) - \frac{(\rho_{int}^i)^2}{2} \le \hat{C} \de^{\frac{\al}{2}} \quad \text{in } B_{\rho_{int}^i} (z_i), \quad \text{ where } \hat{C}:= 2 + 2 d_\Om^2 C_2 \left(  C_2 + 1 \right) .
\end{equation} 

By comparing $u$ -- that is the solution to \eqref{eq:torsionproblem} -- with $v:=\frac{ |x - z_i|^2 - (\rho_{ext}^i)^2 }{2} - 2 \ve$ -- that is the solution to $\De v=N$ in $B_{\rho_{ext}^i} (z_i)$, $v=-2\ve$ on $\pa B_{\rho_{ext}^i} (z_i)$ -- we obtain that
\begin{equation*}
	-u(z_i) \le \frac{(\rho_{ext}^i)^2}{2} + 2 \ve  \text{ in } \Om_{2\ve}^i ,
\end{equation*}
which combined with \eqref{eq:rho_ext-rho_int} leads to 
\begin{equation*}
	-u(z_i) \le \frac{(\rho_{int}^i)^2}{2} + 2 \ve + 
	\frac{ ( 2 d_\Om C_2)^2}{2} \, \de^\al + 2 d_\Om C_2 \rho_{int}^i \, \de^{\frac{\al}{2}} ,
\end{equation*}
and hence, using \eqref{eq:def vepsilon}, \eqref{eq:delta small minore di 1}, and $\rho_{int}^i \le \rho_{ext}^i \le d_\Om$,
\begin{equation*}
	-u(z_i) \le \frac{(\rho_{int}^i)^2}{2} + \left[2 + 2 d_\Om^2 C_2 \left(  C_2 + 1 \right) \right] \de^{\frac{\al}{2}} .
\end{equation*}

On the other hand, by comparing $u$ with $\frac{ |x - z_i|^2 - (\rho_{int}^i)^2 }{2}$ in $B_{\rho_{int}^i} (z_i)$ we immediately obtain that
\begin{equation*}
\frac{(\rho_{int}^i)^2}{2} \le -u(z_i) \quad \text{in } B_{\rho_{int}^i} (z_i).
\end{equation*}
The last two inequalities clearly give \eqref{eq:comparison estimates u(zi) rho_i}.

The triangle inequality, \eqref{eq:stimare integral -u starting eq}, and the fact that $(-u) \ge 0$ give that
\begin{equation*}
	\begin{split}
		\left| \int_{\Om} (-u) dx - \frac{ | B_1 | }{N+2} \sum_{i \in \cI} \left( \rho_{int}^i \right)^{N+2} \right|
		\le &
		\int_{\Om \setminus \bigcup_{i \in \cI} B_{\rho_{int}^i}(z_i) }	(-u) dx
		\\
		& +
		\left| \int_{\bigcup_{i \in \cI} B_{\rho_{int}^i}(z_i) }	h_{z_i} dx \right|
		\\
		& +
		\left| \int_{\bigcup_{i \in \cI} B_{\rho_{int}^i}(z_i) }	(- q_{z_i}) dx - \frac{ | B_1 | }{N+2} \sum_{i \in \cI} \left( \rho_{int}^i \right)^{N+2} \right| ,
	\end{split}
\end{equation*}
from which we easily obtain \eqref{eq:fundamental inequality 2} with
$$
C_4:= |\Om| \left( 2  + 2 G \, d_\Om \, C_2 +  \frac{ d_\Om^2}{2} C_2+  \hat{C} \right)
$$
by using \eqref{eq:term I for step 3}, \eqref{eq:term II for step 3}, \eqref{eq:term IIIA for step 3}, and \eqref{eq:comparison estimates u(zi) rho_i}.

{\bf Step 5.} Combining the previous two steps, in particular \eqref{eq:fundamental inequality 1} and \eqref{eq:fundamental inequality 2}, we easily obtain that
\begin{equation}\label{eq:inequality radii squared}
	\left| R^2 |\Om| - |B_1| \sum_{i \in \cI} (\rho_{int}^i)^{N+2} \right| \le (N+2) \left( C_3 + C_4 \right) \, \de^{\frac{\al}{2}} ,
\end{equation}
where we also used \eqref{eq:delta small minore di 1} as usual.
We now also shall prove that
\begin{equation}\label{eq:inequality radii}
	\left| R |\Om| - |B_1| \sum_{i \in \cI} (\rho_{int}^i)^{N+1} \right| \le C_5 \, \de^{\frac{\al}{2}} 
\end{equation}
so that
\begin{equation}\label{eq:radii sum inequality}
	\begin{split}
	|B_1| \sum_{i \in \cI} (\rho_{int}^i)^{N} \left( \rho_{int}^i - R \right)^2  
	& = |B_1| \sum_{i \in \cI} (\rho_{int}^i)^{N+2} + R^2 |B_1| \sum_{i \in \cI} (\rho_{int}^i)^{N} - 2 R |B_1| \sum_{i \in \cI} (\rho_{int}^i)^{N+1}
	\\
	& \le |B_1| \sum_{i \in \cI} (\rho_{int}^i)^{N+2} + R^2 |\Om| - 2 R |B_1| \sum_{i \in \cI} (\rho_{int}^i)^{N+1}
	\\
	& \le C_6 \, \de^{\frac{\al}{2}}  \qquad \qquad \quad \text{ with } C_6:= (N+2)(C_3 + C_4) + 2 G \, C_5,
\end{split}
\end{equation}
where the first inequality follows from the fact that $B_{\rho_{int}^i} (z_i)$ are disjoint balls all contained in $\Om$, whereas the last inequality follows from \eqref{eq:inequality radii squared}, \eqref{eq:inequality radii}, and \eqref{eq:stima da sopra di R}.

In order to prove \eqref{eq:inequality radii}, we notice that, since
$$\dv \left(|\na u| \na u \right) = (N+1 ) | \na u| + |\na u|^{-1} \langle (D^2 u - I) \na u, \na u \rangle $$
the divergence theorem gives that
\begin{equation}\label{eq:label per correzione a step 5}
	\begin{split}
	(N+1) \int_{\Om} |\na u| dx 
	& = \int_{\pa\Om} | \na u | u_\nu \, dS_x - \int_\Om |\na u|^{-1} \langle (D^2 u - I) \na u, \na u \rangle dx  
	\\
	& = R N |\Om| + \int_{\pa\Om} \left(  u_\nu^2 - R^2 \right) dS_x - \int_\Om |\na u|^{-1} \langle (D^2 u - I) \na u, \na u \rangle dx ,
	\end{split}
\end{equation}
where the last identity immediately follows recalling the definition of $R$ in \eqref{eq:def H_0 and R}; hence, using \eqref{eq:gradient bound}, the H\"{o}lder inequality, and \eqref{eq:CS deficit and positive part deviation} to see that
\begin{equation*}
\begin{split}
\left| \int_\Om |\na u|^{-1} \langle (D^2 u - I) \na u, \na u \rangle dx \right| 
& \le \int_\Om \left| D^2 u - I \right| | \na u | dx
\\
& \le G |\Om|^{\frac{1}{2}} \left( \int_\Om \left| D^2 u - I \right|^2 dx \right)^{\frac{1}{2}}
\\
& = G |\Om|^{\frac{1}{2}} \left( \int_\Om \left( |D^2 u|^2 - \frac{(\De u)^2}{N} \right) dx  \right)^{\frac{1}{2}}
\\
& \le \sqrt{N-1} \, G^2 |\Om|^{\frac{1}{2}} \, \de^{\frac{1}{2}} ,
\end{split}
\end{equation*}
and using \eqref{eq:Serrin deficit and positive part deviation} (similarly to \eqref{eq:AA la richiamo per dopo senza diametro} but without $\langle x-z , \nu \rangle$ and $d_\Om$) to see that
\begin{equation*}
\left| \int_{\pa\Om} \left(  u_\nu^2 - R^2 \right) dS_x \right| \le 2 \, \left( N | \Om |\right)^{\frac{1}{2}}  \, G^2 \, \de^{\frac{1}{2}} , 	
\end{equation*}
\eqref{eq:label per correzione a step 5} easily leads to
\begin{equation}\label{eq:step 5 I}
	\left| R |\Om| - \frac{(N+1)}{N} \int_{\Om} |\na u| dx \right| \le \left( \frac{2 \sqrt{N} + \sqrt{N-1} }{N} \right) | \Om |^{\frac{1}{2}} \, G^2 \, \de^{\frac{1}{2}} .
\end{equation}
On the other hand, recalling \eqref{eq:star per comodit alla fine} we have
\begin{equation}\label{eq:step5prova}
		\int_{\Om} |\na u | dx -  \sum_{i \in \cI} \int_{ B_{\rho_{int}^i}(z_i) }	| \na q_{z_i}|  dx =
		\int_{\Om \setminus \bigcup_{i \in \cI} B_{\rho_{int}^i}(z_i) }	| \na u | dx
	+
		\sum_{i \in \cI} \int_{ B_{\rho_{int}^i}(z_i) }	\left( |\na u| - | \na q_{z_i}| \right)  dx ,
\end{equation}
and hence the triangle inequality gives
\begin{equation}\label{eq:NEW_gradient from above and below step 5}
\left| \int_{\Om} |\na u | dx - \sum_{i \in \cI} \int_{ B_{\rho_{int}^i}(z_i) }	| \na q_{z_i}  | dx \right| 	\le \int_{\Om \setminus \bigcup_{i \in \cI} B_{\rho_{int}^i}(z_i) }	| \na u | dx
+
\sum_{i \in \cI} \int_{ B_{\rho_{int}^i}(z_i) }	| \na h_{z_i} | dx ,
\end{equation}
where $q_{z_i}$ and $h_{z_i}$ are those defined in \eqref{eq:def h_z_i}.

By direct computation it is easy to check that
\begin{equation*}
	\int_{B_{\rho_{int}^i}(z_i)}  |x - z_i| dx
	=
	\frac{1}{N+1}\int_{\pa B_{\rho_{int}^i}(z_i)} |x-z_i| \langle x- z_i , \nu \rangle dS_x 
	=
	\frac{ N | B_1 | \left( \rho_{int}^i \right)^{N+1} }{N+1} ,
\end{equation*} 
so that, recalling the definition of $q_{z_i}$ in \eqref{eq:def h_z_i},
\begin{equation}\label{eq:step5 term 2}
\sum_{i \in \cI}	\int_{ B_{\rho_{int}^i}(z_i) }	|\na q_{z_i} | dx = \frac{ N |  B_1 | }{N+1} \sum_{i \in \cI} \left( \rho_{int}^i \right)^{N+1} .
\end{equation}
We now estimate the terms appearing in the the right-hand side of \eqref{eq:NEW_gradient from above and below step 5}.
Using \eqref{eq:Taylor for the gradient of h_{z_i}} and recalling
\eqref{eq:star per comodit alla fine}, we obtain that
\begin{equation}\label{eq:step5 term 1}
	\sum_{i \in \cI} \int_{ B_{\rho_{int}^i}(z_i) }	| \na h_{z_i} | dx  \le d_\Om \, \sqrt{N} \, C_2 |\Om| \,  \de^{\frac{\al}{2}}.
\end{equation}

Finally, for the first summand on the right-hand side of \eqref{eq:NEW_gradient from above and below step 5}, we combine 
\eqref{eq:stima misura complementare di tutte le palle} and the definition of $G$ in \eqref{eq:gradient bound} to obtain that
\begin{equation}\label{eq:step5 term 3}
	\int_{\Om \setminus \bigcup_{i \in \cI} B_{\rho_{int}^i}(z_i) }	| \na u | dx \le 
	G \, \left| \Om \setminus \bigcup_{i \in \cI} B_{\rho_{int}^i}(z_i) \right| 
	\le C_7 \, \de^{\frac{\al}{2}},
\end{equation} 
where
\begin{equation*}
C_7 := G \, \left\lbrace \left[ 1 + \left( 2 + 2 d_\Om \, C_2 \right) \cM_0^- \right]^{N-1}  \left( 2 + 2 d_\Om \, C_2 \right)   \right\rbrace \, N \, | \Om | ;
\end{equation*}
here we used that by definition \eqref{eq:def H_0 and R}, $|\pa\Om|=N|\Om|/R \le N |\Om|$, being as $R \ge 1$.

Putting together \eqref{eq:NEW_gradient from above and below step 5}, \eqref{eq:step5 term 2}, \eqref{eq:step5 term 1}, and \eqref{eq:step5 term 3} we thus obtain that
\begin{equation}\label{eq:step 5 II}
 \left| \frac{(N+1)}{N} \int_{\Om} |\na u | dx 	-  |  B_1 |  \sum_{i \in \cI} \left( \rho_{int}^i \right)^{N+1} \right| \le C_8 \, \de^{\frac{\al}{2}} , \, \text{ where } C_8 := \frac{(N+1)}{\sqrt{N}} \left( d_\Om C_2 |\Om| + \frac{C_{7}}{\sqrt{N}} \right).
\end{equation}
 
Combining \eqref{eq:step 5 I}, \eqref{eq:step 5 II}, and, as usual, \eqref{eq:delta small minore di 1}, we thus obtain \eqref{eq:inequality radii} with
\begin{equation*}
	C_5:= \left( \frac{2 \sqrt{N} + \sqrt{N-1} }{N} \right) | \Om |^{\frac{1}{2}} \, G^2 + C_8 .
\end{equation*}

{\bf Step 6.} 
By \eqref{eq:radii sum inequality}, we deduce that
\begin{equation*}
	|B_1| (\rho_{int}^i)^{N} \left( \rho_{int}^i - R \right)^2 \le C_6 \, \de^{\frac{\al}{2}} \quad \text{ for any } i \in \cI .
\end{equation*}
Therefore, for any $i \in \cI$, analyzing the two cases $\rho_{int}^i<R/2$ and $\rho_{int}^i \ge R/2$ we find that either
\begin{equation}\label{eq:small radii}
	\rho_{int}^i \le \frac{C_9}{R^{\frac{2}{N}}} \, \de^{\frac{\al}{2 N}} \le  C_9 \, \de^{\frac{\al}{2 N}}
\end{equation}
or
\begin{equation}\label{eq:radii big enough}
	| \rho_{int}^i - R| \le \frac{C_{10}}{R^{\frac{N}{2}}} \,  \de^{\frac{\al}{4}} \le C_{10} \,  \de^{\frac{\al}{4}} ;
\end{equation}
the last inequalities in \eqref{eq:small radii} and \eqref{eq:radii big enough} easily follow recalling that $R \ge 1$.
It is easy to check that \eqref{eq:small radii} and \eqref{eq:radii big enough} hold true with
\begin{equation*}
	C_9 := \left( \frac{ 4 C_6 }{|B_1|} \right)^{\frac{1}{N}} 
	\quad \text{ and } \quad 
	C_{10} := \left( \frac{ 2^N C_6 }{|B_1|} \right)^{\frac{1}{2}} .
\end{equation*}
Set
\begin{equation*}
\cB := \left\lbrace i \in \cI \, : \, \rho_{int}^i \text{ satisfies } \eqref{eq:radii big enough} \right\rbrace \subset \cI  \quad \text{ and } m:=| \cB| .
\end{equation*}
Notice that $m$ is a finite natural number; in fact, by \eqref{eq:radii big enough},
\begin{equation}\label{eq:in proof bound m}
	m |B_1| \left( R- C_{10} \,  \de^{\frac{\al}{4}} \right)^N \le |B_1| \sum_{ i\in \cB} ( \rho_{int}^i )^N \le \left| \Om \right| ,  
\end{equation}
that is, \eqref{eq:bound on number m}.

By \eqref{eq:inequality radii squared}, we compute
\begin{equation*}
	\begin{split}
	R^2 |\Om| 
	& \le | B_1 | \sum_{ i \in \cI } ( \rho_{int}^i )^{N+2} + \left( C_3 + C_4 \right) \, \de^{\frac{\al}{2}} 
	\\
	& = | B_1 | \sum_{ i \in \cB } ( \rho_{int}^i )^{N+2} + | B_1 | \sum_{ i \in \cI \setminus \cB } ( \rho_{int}^i )^{N+2} +  \left( C_3 + C_4 \right) \, \de^{\frac{\al}{2}}
	\\
	& \le | B_1 | \sum_{ i \in \cB } ( \rho_{int}^i )^{N+2} +   \left( 
	C_9 
	\, \de^{\frac{\al}{2 N}} \right)^2  \left| \bigcup_{i \in \cI \setminus \cB} B_{\rho_{int}^i} (z_i) \right| +  \left( C_3 + C_4 \right) \, \de^{\frac{\al}{2}}
	\\
	& \le | B_1 | \sum_{ i \in \cB } ( \rho_{int}^i )^{N+2} +    
	C_9^2 
	\,   \left| \Om \right| \, \de^{\frac{\al}{N}} +  \left( C_3 + C_4 \right) \, \de^{\frac{\al}{2}}
	\\
	& \le \left( R + 
	C_{10} 
	\,  \de^{\frac{\al}{4}}  \right)^2 \left| \bigcup_{i \in \cB} B_{\rho_{int}^i} (z_i)  \right| +   
	C_9^2  
	\,   \left| \Om \right| \, \de^{\frac{\al}{N}} +  \left( C_3 + C_4 \right) \, \de^{\frac{\al}{2}} 
	\\
	& \le 
	R^2 \left| \bigcup_{i \in \cB} B_{\rho_{int}^i} (z_i)  \right| + \left[
	\left(
	C_{10} 
	\,  \de^{\frac{\al}{4}}  \right)^2 + 2 \, R \, C_{10} \, \de^{\frac{\al}{4}} \right] \left| \Om  \right| +   
	C_9^2  
	\,   \left| \Om \right| \, \de^{\frac{\al}{N}} +  \left( C_3 + C_4 \right) \, \de^{\frac{\al}{2}}
	,
	\end{split}
\end{equation*}
where we used \eqref{eq:small radii}, \eqref{eq:star per comodit alla fine}, and \eqref{eq:radii big enough}. Recalling \eqref{eq:delta small minore di 1} and \eqref{eq:stima da sopra di R}, we thus obtain that
\begin{equation*}
		R^2 \left( |\Om| - \left| \bigcup_{i \in \cB} B_{\rho_{int}^i } (z_i)  \right| \right) \le
		 C_{11} \,  \de^{\frac{\al}{\max \left\lbrace 4, N \right\rbrace}}  
		 \quad \text{ with }
		 C_{11} := |\Om| \left( 
		 C_{10}^2 
		 + 2 
		 G \, C_{10} 
		 + 
		 C_9^2 
		 \right)
		 + \left( C_3 + C_4 \right) , 
\end{equation*}
and hence
\begin{equation}\label{eq:inequality for Brhoi}
\left| \Om \De \bigcup_{i \in \cB} B_{\rho_{int}^i } (z_i)  \right| = \left| \Om \setminus \bigcup_{i \in \cB} B_{\rho_{int}^i} (z_i)  \right| =	|\Om| - \left| \bigcup_{i \in \cB} B_{\rho_{int}^i} (z_i)  \right| \le \frac{C_{11}}{R^2}  \, \de^{\frac{\al}{\max \left\lbrace 4, N \right\rbrace}}
\le C_{11}  \, \de^{\frac{\al}{\max \left\lbrace 4, N \right\rbrace}} ,
\end{equation}
being as $R \ge 1$.
We thus have found $m$ disjoint balls $B_{\rho_{int}^i } (z_i) \subset \Om$ satisfying \eqref{eq:radii big enough}, \eqref{eq:inequality for Brhoi}, and \eqref{eq:in proof bound m}. 
Clearly, without loss of generality, we can assume that $\cB = \left\lbrace 1, \dots, m \right\rbrace$, and hence,
setting
\begin{equation*}
	\mathcal{F}:= \bigcup_{i \in \cB} B_{\rho_{int}^i} (z_i) = \bigcup_{i =1}^m B_{\rho_{int}^i} (z_i),
\end{equation*}
\eqref{eq:thm_new global statement radii}, \eqref{eq:thm_new global statement measure}, and \eqref{eq:bound on number m} immediately follow.
Moreover, from \eqref{eq:inclusion tubular} we deduce that, for any $i \in \cB$,
$$
\mathrm{dist}_{\pa\Om}(x) \le \left( 2 + 2 \, d_\Om \, C_2 \right) \, \de^{\frac{\al}{2}} \quad \text{ for any } x \in \pa B_{\rho^i_{int}}(z_i) ,
$$
from which \eqref{eq:thm_one-sided Hausdorff closeness} easily follows.

{\bf Step 7.} It only remains to prove the closeness of the perimeter of $\mathcal{F}$ to that of $\Om$, that is \eqref{eq:thm_perimeter closeness}. This can be easily deduced from the estimates obtained in the previous steps, using an argument similar to that used at the end of \cite[{\it Step 3} of the proof of Theorem 1.1]{CirMag}. 
To this aim, recalling the definition of $R$ in \eqref{eq:def H_0 and R}, we compute that
\begin{equation}\label{eq:newproofforperimetercloseness_1}
	\left| | \pa\Om | - | \pa \mathcal{F} | \right| 
	= \left| \frac{N |\Om| }{R} - | \pa \mathcal{F} | \right| 
	\le \frac{N}{R} \left| \left| \Om \right| - \left| \mathcal{F} \right| \right| + \left| \frac{N|\mathcal{F}|}{R} - | \pa \mathcal{F} | \right| .
\end{equation}

Recalling that $R \ge 1$ and using \eqref{eq:inequality for Brhoi} we easily compute that
\begin{equation}\label{eq:newproofforperimetercloseness_2}
	\frac{N}{R} \left| \left| \Om \right|  - \left| \mathcal{F} \right| \right| \le N \left| \left| \Om \right|  - \left| \mathcal{F} \right| \right| \le N \, C_{11} \, \de^{\frac{\al}{\max \left\lbrace 4, N \right\rbrace}} .
\end{equation}

Recalling that $R\ge 1$, from \eqref{eq:radii big enough} it is easy to see that 
\begin{equation}\label{eq:thanks to 3 smallness rho_int^i >= R/2}
	\rho_{int}^i \ge \frac{R}{2} ,
\end{equation}
provided that
\begin{equation*}
	\de < \left( \frac{1}{2 \, C_{10}} \right)^{\frac{4}{\al}} ;
\end{equation*}
hence, using \eqref{eq:thanks to 3 smallness rho_int^i >= R/2}, $R\ge 1$, and \eqref{eq:star per comodit alla fine} we compute that 
\begin{equation}\label{eq:newproofforperimetercloseness_3}
	\begin{split}
		\left| \frac{N|\mathcal{F}|}{R} - | \pa \mathcal{F} | \right| 
		& = N |B_1| \left| \sum_{i \in \cB} \left( \rho_{int}^i \right)^{N-1} \left( \frac{\rho_{int}^i}{R} -1 \right) \right|
		\\
		& \le N  \left( \max_{ i \in \cB} \left| \rho_{int}^i - R \right| \right) |B_1| \sum_{i \in \cB}  \left( \rho_{int}^i \right)^{N-1}
		\\
		& \le 2 \, N  \left( \max_{ i \in \cB} \left| \rho_{int}^i - R \right| \right) |B_1| \sum_{i \in \cB}  \left( \rho_{int}^i \right)^{N}
		\\
		& \le 2 \, N  | \Om | \max_{ i \in \cB} \left| \rho_{int}^i - R \right| 
		\\
		& \le 2 \, N  | \Om | C_{10} \,  \de^{\frac{\al}{4}} ,
	\end{split}
\end{equation}
%
where the last inequality follows from \eqref{eq:radii big enough}.

Combining \eqref{eq:newproofforperimetercloseness_1}, \eqref{eq:newproofforperimetercloseness_2}, \eqref{eq:newproofforperimetercloseness_3}, and recalling \eqref{eq:delta small minore di 1} we easily obtain that
\begin{equation*}
	\left| | \pa\Om | - | \pa \mathcal{F} | \right| \le N \left( C_{11}+ 2 | \Om | C_{10} \right) \, \de^{\frac{\al}{\max \left\lbrace 4, N \right\rbrace}} ,
\end{equation*}
that is, \eqref{eq:thm_perimeter closeness}.
This completes the proof.
%
\end{proof}

We now proceed with the
\begin{proof}[Proof of Corollary \ref{cor:two corollaries}]
	Item (i) immediately follows combining Theorem \ref{thm:stability with diameter} and Remark \ref{rem:intro}.
	
	For item (ii) we simply notice that, by the triangle inequality
	\begin{equation*}
		\nr H \nr_{L^{N-2} (\pa \Om)} \le \nr H_0 \nr_{L^{N-2} (\pa \Om)}  +  \nr H - H_0 \nr_{L^{N-2} (\pa \Om)} \le  N^{-1} C_0^{\frac{2N-3}{N-2}}  +  \nr H - H_0 \nr_{L^{N-2} (\pa \Om)} ,
	\end{equation*}
	where we used the definition \eqref{eq:def H_0 and R} and that $|\pa\Om|\le C_0$ and $|\Om| \ge 1/C_0$;
	hence,
	\begin{equation*}
		\int_{\pa\Om} |H|^{N-2} \, dS_x  \le \left( N^{-1} C_0^{\frac{2N-3}{N-2}}  +  1 \right)^{ N-2 }
	\end{equation*}
	provided that
	\begin{equation*}
	\nr H - H_0 \nr_{L^{N-2} (\pa \Om)} \le 1 .
	\end{equation*}
	The desired result then immediately follows from item (i) of Corollary \ref{cor:two corollaries}, noting that
	\begin{equation*}
	\de := \int_{\pa\Om} \left(H_0-H\right)^+ dS_x \le | \pa \Om |^{\frac{N-3}{N-2}} \nr H_0 - H \nr_{L^{N-2} (\pa \Om)} \le C_0^{\frac{N-3}{N-2}} \nr H_0 - H \nr_{L^{N-2} (\pa \Om)} ,
	 \end{equation*}
by the H\"{older} inequality and the assumption $|\pa\Om|\le C_0$.
\end{proof}


We conclude with the

\begin{proof}[Proof of Theorem \ref{thm:NewAlternative_equal radii}]
Notice that the analogue of (i) of Remark \ref{rem:Initial remarks on proof of THM} also holds true for Theorem \ref{thm:NewAlternative_equal radii}. Thus, it is enough to prove the desired result when $\de \le c$, where $c$ is some explicit constant only depending on the parameters indicated in the statement, i.e., $N$, $d_\Om$, and $\cM_0^-$. Hence, we can assume that \eqref{eq:delta small minore di 1} is in force. 

Let $z_i$ and $\rho_{int}^i$, $i = 1, \dots, m$, be those obtained in
Theorem \ref{thm:stability with diameter}.
Recalling that $R \ge 1$, from \eqref{eq:thm_new global statement radii} and \eqref{eq:bound on number m} it is easy to check that
\begin{equation}\label{eq:new_explicit R/2 lower bound for rho_int^i in proof of the last Theorem}
	\rho_{int}^i \ge \frac{R}{2} \quad \text{ for any } i=1,\dots,m
\end{equation}
and
\begin{equation}\label{eq:new_explicit bound for m under smallness 3}
 m \le 2^N \frac{|\Om|}{|B_1|} ,
\end{equation}
provided that $\de < \left( \frac{1}{2C} \right)^{\frac{4}{\al}} $ (where $C$ is as in \eqref{eq:thm_new global statement radii} and \eqref{eq:bound on number m}).

Given the set
$$\mathcal{F}:= \bigcup_{i=1}^m B_{\rho_{int}^i} (z_i)$$
defined in Theorem \ref{thm:stability with diameter}, we can repeat the iterative argument from \cite[{\it Step 5} of the proof of Theorem~1.1]{CirMag} to obtain $m$ points $\widehat{z}_i$
such that the balls $B_R (\widehat{z}_i)$ (centered at $\widehat{z}_i$, with radius $R$), $i = 1, \dots, m$, are disjoint and the set
\begin{equation*}
	\widehat{\mathcal{F}}:= \bigcup_{i=1}^m B_{R} (\widehat{z}_i)
\end{equation*}
satisfies
\begin{equation}\label{eq:in proof_NewAlternative_thm_for L^1 closeness}
	\left| \mathcal{F} \De \widehat{\mathcal{F}} \right| \le C \,  \de^{ \frac{\al}{ 4 }} ,
\end{equation}
\begin{equation}\label{eq:in proof_NewAlternative_thm_for one-sided Hausdorff closeness}
	\max_{x\in \pa \widehat{\mathcal{F}} } \mathrm{dist}_{\pa \mathcal{F} } (x) \le C \, \de^{\frac{\al}{4}},
\end{equation}
thanks to \eqref{eq:thm_new global statement radii} and \eqref{eq:new_explicit bound for m under smallness 3}.
Here, the constants $C$ can be explicitly computed and only depend on $N$, and upper bounds on $d_\Om$ and $\cM_0^-$.
 
Since $\mathcal{F} \subset \Om$, we have that
\begin{equation*}
	\left| \Om \De \widehat{\mathcal{F}} \right| \le |\Om \setminus \mathcal{F}| + \left| \mathcal{F} \De \widehat{\mathcal{F}} \right| ,
\end{equation*}
and hence \eqref{eq:NewAlternative_thm_new global statement measure} easily follows using \eqref{eq:thm_new global statement measure} and \eqref{eq:in proof_NewAlternative_thm_for L^1 closeness}, recalling \eqref{eq:delta small minore di 1}.

Noting that
\begin{equation*}
\max_{x\in \pa \widehat{\mathcal{F}} } \mathrm{dist}_{\pa\Om} (x) \le \max_{x\in \pa \mathcal{F} } \mathrm{dist}_{\pa\Om} (x) + \max_{x\in \pa \widehat{\mathcal{F}} } \mathrm{dist}_{\pa \mathcal{F} } (x) ,
\end{equation*}
\eqref{eq:NewAlternative_thm_one-sided Hausdorff closeness} easily follows using \eqref{eq:thm_one-sided Hausdorff closeness} and \eqref{eq:in proof_NewAlternative_thm_for one-sided Hausdorff closeness}, recalling \eqref{eq:delta small minore di 1}. 

Finally, to establish \eqref{eq:NewAlternative_thm_perimeter closeness}, we notice that
\begin{equation}\label{eq:proofnewAlternative for perimeter_1}
	\left| | \pa \Om| - | \pa \widehat{\mathcal{F}} | \right| \le 	\left| | \pa \Om| - | \pa \mathcal{F} | \right| + 	\left| | \pa \mathcal{F} | - | \pa \widehat{\mathcal{F}} | \right| ,
\end{equation}
and compute that
\begin{equation}\label{eq:proofnewAlternative for perimeter_2}
	\begin{split}
		\left| | \pa \mathcal{F} | - | \pa \widehat{\mathcal{F}} | \right|
		& \le N |B_1| \sum_{i=1}^m \left| (\rho_{int}^i)^{N-1} - R^{N-1} \right| 
		\\
		& \le  N (N-1) |B_1| \sum_{i=1}^m \max\left\lbrace \rho_{int}^i , R \right\rbrace^{N-2} \left| \rho_{int}^i - R \right|
		\\
		& \le 2^{N-2} \, N (N-1)  \left( \max_{i\in\left\lbrace 1, \dots , m \right\rbrace } \left| \rho_{int}^i - R \right| \right) |B_1| \sum_{i=1}^m (\rho_{int}^i)^{N-2} 
		\\
		& \le 2^{N} \, N (N-1) |\Om| \max_{i\in\left\lbrace 1, \dots , m \right\rbrace } \left| \rho_{int}^i - R \right| ,
	\end{split}	 
\end{equation}
where we used Lagrange's mean value theorem, \eqref{eq:new_explicit R/2 lower bound for rho_int^i in proof of the last Theorem}, and 
$$
|B_1| \sum_{i=1}^m (\rho_{int}^i)^{N-2} \le 2^2 |B_1| \sum_{i=1}^m (\rho_{int}^i)^{N} \le 2^2 |\Om|, 
$$
which follows from \eqref{eq:new_explicit R/2 lower bound for rho_int^i in proof of the last Theorem}, $R \ge 1$, and \eqref{eq:star per comodit alla fine}.
Combining \eqref{eq:proofnewAlternative for perimeter_1}, \eqref{eq:thm_perimeter closeness}, \eqref{eq:proofnewAlternative for perimeter_2}, \eqref{eq:thm_new global statement radii}, and recalling \eqref{eq:delta small minore di 1} we easily obtain \eqref{eq:NewAlternative_thm_perimeter closeness}.
\end{proof}


\section*{Acknowledgements}
The author is supported by the Australian Research Council (ARC) Discovery Early Career Researcher Award (DECRA) DE230100954 ``Partial Differential Equations: geometric aspects and applications'' and by the Australian Academy of Science J G Russell Award. 
The author is member of the Australian Mathematical Society (AustMS).

The author thanks Francesco Maggi for interesting comments and Emanuel Indrei for pointing out that some of the techniques used in the present paper may also be useful in relation to \cite{E.Indrei}.
The author also thanks the anonymous referees for their careful reading and helpful comments. In particular, the author acknowledges the suggestions of one of the referees to add Theorem \ref{thm:NewAlternative_equal radii} and to include the one-sided Hausdorff estimate \eqref{eq:thm_one-sided Hausdorff closeness} and the perimeter bound \eqref{eq:thm_perimeter closeness} in the statement of Theorem \ref{thm:stability with diameter}.



	\end{document}